\documentclass{amsart}
\usepackage[utf8]{inputenc}
  \usepackage{amsmath}
\usepackage{enumerate}
  \usepackage{amssymb}
  \usepackage{amsthm}
  \usepackage{multirow}
\usepackage{mathtools}
  \usepackage{appendix}
  \usepackage{bm}
  \usepackage{graphicx}
  \usepackage{color}
  \usepackage{epigraph}
  \usepackage{bbold}
\usepackage{epsfig}
\usepackage{amsmath,amssymb}
\hoffset - 0.5 in
\newtheorem{Th}{Theorem}[section]
\newtheorem{lem}[Th]{Lemma}

\theoremstyle{definition}
\newtheorem{Def}[Th]{Definition}

\newtheorem{Cor}[Th]{Corollary}
\newtheorem{Prop}[Th]{Proposition}

\theoremstyle{remark}
\newtheorem*{rem}{\bf Remark}
\newtheorem*{que}{\bf Questions}
\newtheorem{Conj}{\bf Conjecture of Chowla}
\newtheorem*{thank}{\ \ \ \bf Acknowledgment}

\numberwithin{equation}{section}

\newcommand{\xdownarrow}[1]{%
  {\left\downarrow\vbox to #1{}\right.\kern-\nulldelimiterspace }
}
\newcommand{\tend}[3][]{\xrightarrow[#2\to#3]{#1}}
\newcommand{\tenddw}[1][]{\xdownarrow{#1}}

\newcommand{\egdef}{\stackrel{\textrm {def}}{=}}
\newcommand{\ds}{\displaystyle}

\newcommand{\1}{\mathbb{1}}
\renewcommand{\P}{\mathbb{P}}
\newcommand{\Z}{\mathbb{Z}}
\newcommand{\re}{\textrm {Re}}
\newcommand{\N}{\mathbb{N}}
\newcommand{\T}{\mathbb{T}}
\newcommand{\C}{\mathbb{C}}

\newcommand{\nB}{\mathcal{NB}}

\newcommand{\cl}{\mathcal{L}}

\newcommand{\fr}{\textrm{fr}}
\def \equi#1{\mathrel{\mathop{\kern 0pt\sim}\limits_{#1}}}
\newcommand{\mob}{\boldsymbol{\mu}}
\newcommand{\lamob}{\boldsymbol{\lambda}}
\newcommand{\setdef}{\stackrel {\rm {def}}{=}}

\title[On the Erd\"{o}s flat polynomials problem]{On the Erd\"{o}s flat polynomials problem, Chowla conjecture and Riemann Hypothesis}

\author{e. H. el Abdalaoui${}^{\bigstar}$}
\address{University of Rouen Normandy,
  Department of Mathematics, LMRS  UMR 60 85 CNRS\\
Avenue de l'Universit\'e, BP.12
76801 Saint Etienne du Rouvray - France .}
\email{elhoucein.elabdalaoui@univ-rouen.fr}

\urladdr{http://www.univ-rouen.fr/LMRS/Persopage/Elabdalaoui/}

\date{\today}
\subjclass[2010]{Primary 42A05, 42A55, 05D99, Secondary 37A05, 37A30.}
\dedicatory{\setlength\epigraphrule{0pt}\epigraph{\textbf{``The hardest thing of all is to find a 
black cat in a dark room, especially if there is no cat.''}}{{---Confucius}}}

\keywords{ Merit factor, flat polynomials, ultraflat polynomials, Erd\"{o}s-Newman flatness problem, 
Littlewood flatness problem, digital transmission, Morse cocycle, Turyn-Golay's conjecture, simple Lebesgue spectrum, 
Banach problem, Banach-Rhoklin problem, singular spectrum, Barker sequences, Morse sequences, Liouville function, M\"{o}bius function,
Chowla conjecture, Riemann Hypothesis (RH).\\
\rule{0.2\textwidth}{.4pt}\hfill \rule{0.2\textwidth}{.4pt} 
\vskip 0.1cm 
$${\bigstar}_{\includegraphics[scale=0.25]{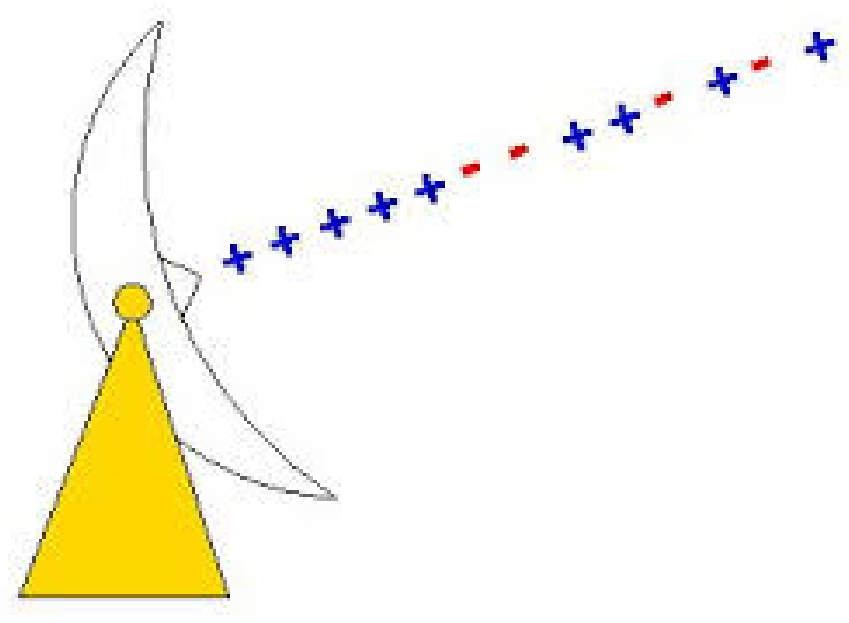}}
$$
}
\begin{document}
\begin{abstract} 
There are no square $L^2$-flat sequences of polynomials of the type
  $$\frac{1}{\sqrt q}( \epsilon_0 + \epsilon_1z + \epsilon_2z^2 + \cdots + \epsilon_{q-2}z^{q-2} +\epsilon_q z^{q-1}),$$
where for each $j,~~ 0 \leq j\leq q-1,~\epsilon_j = \pm 1$. It follows that Erd\"{o}s's conjectures on Littlewood polynomials hold. 
Consequently, Turyn-Golay's conjecture is true, that is, there are only finitely many Barker sequences. 
We further get that the spectrum of dynamical systems arising from continuous Morse sequences is singular. 
This settles an old question due to M. Keane. Applying our reasoning to the Liouville function we obtain
that the popular Chowla conjecture on the 
normality of the Liouville function implies Riemann hypothesis.
\end{abstract}
\maketitle
\section{Introduction.}
The main purpose of this paper is to establish that ~there are no square $L^2$-flat sequences of polynomials of the type
\begin{eqnarray}\label{Littlewood}
  \frac{1}{\sqrt q}( \epsilon_0 + \epsilon_1z + \epsilon_2z^2 + \cdots + \epsilon_{q-2}z^{q-2} + \epsilon_{q-1} z^{q-1}),
\end{eqnarray}
where for each $j,~~~ 0 \leq j\leq q-1,~\epsilon_j = \pm 1$ and $z \in S^1,$  $S^1$ is the circle group. It follows that there are
only finitely many Barker sequences. We thus get an affirmative answer to Turyn-Golay's conjecture and Erd\"{o}s's conjectures. 
Furthermore, our result implies that the spectrum of dynamical system arising from generalized Morse sequences is singular for 
every continuous Morse sequence. This answer an old question due to M. Keane \cite{Keane}.
\vskip 0.1cm


We remind that Turyn-Golay's conjecture, arising from the digital communications engineering, 
state that the merit factor of any binary sequence is bounded. \\

The merit factor of a binary sequence ${\boldsymbol{\epsilon}}=(\epsilon_j)_{j=0}^{n-1} \in \{-1,1\}^n$ is given by
$$F=F_n=F_n({\boldsymbol{\epsilon}})\egdef\frac{1}{\Big\|P_n\Big\|_4^4-1},$$
where
$P_n(z)=\frac1{\sqrt{n}}\sum_{j=0}^{n-1}\epsilon_j z^j,~~z \in S^1$.
Clearly, Turyn-Golay's conjecture is equivalent to $L^4$ conjecture of Erd\"{o}s which say that for any polynomial $P$ 
from the type \eqref{Littlewood} we have $\big\|P\big\|_4 \geq (1+c),$ for some absolutely constant $c>0$. This conjecture implies the well known ultraflat conjecture of Erd\"{o}s which state that for any polynomial $P$ from the type 
\eqref{Littlewood}  we have $\big\|P\big\|_{\infty} \geq (1+c),$ $c>0$. In the same spirit, Newman mentioned (without attribution) that it is 
conjectured that there is a constant $c'>0$ such that for any polynomial $P$ of type \eqref{Littlewood} we have $\big\|P\big\|_1 \leq c'<1$. This conjecture is nowadays 
known as $L^1$ Newman's conjecture. It is obvious that Newman's conjecture implies
 the two  conjectures of Erd\"{o}s's \cite{ErdosI},\cite[Problem 22]{ErdosII}, \cite{ErdosIII}.
\vskip 0.1cm
Of course Newman's conjecture implies also Turyn-Golay conjecture. 
However, our arguments break down as far as the $L^1$-Newman's conjecture is concerned.\\

We notice that our proof is based on the description of some arithmetic set associated to the sequence $(\epsilon_j)$ 
and on the exact computation of the $L^p$ norm of Dirichlet Kernel. We further apply some tools from Gowers's method \cite{Gowers}.\\

Our proof gives also that the square flatness implies the pairwise independence behavior, that is, the canonical projections 
of dynamical system generated by $(\epsilon_j)$ are stochastically pairwise independent. This allows us to reduced our investigation to the 
case of random trigonometric polynomials for which the random coefficients are pairwise independent and to obtain a dynamical proof
of our main result.\\

This dynamical proof confirm partially the heuristic argument of P. Cohen about the behavior of the $L^1$ norms of the exponential sums 
(see the introduction of \cite{Cohen}.). P. Cohen in his paper addressed the famous $L^1$ Littlewood’s conjecture which was solved by 
McGehee-Pigno \& Smith \cite{Smith}.  In our setting, the heuristic argument of P. Cohen can be adapted to infer that the flatness 
may implies normality, that is, if $(P_q)$ is flat in almost everywhere sense then the sequence $(\epsilon_j)$ is normal 
in the following sense: let $k \geq 1$, and $x_n=\pm 1$. Then, for $1 \leq n \leq k$, we have 
$$\frac{\Big|\Big\{ j\in [1,N]~~:~~\epsilon_{j+n}=x_n,~~~n=1,\cdots,k\Big\}\Big|}{N} \tend{N}{+\infty}\frac1{2^k}.$$

Taking into account this interpretation, it follows that if $(\epsilon_j)$ is the Liouville function then the almost everywhere flatness may 
implies the popular Chowla conjecture on the Liouville's function. Roughly speaking, this conjecture state that the Liouville function is 
normal (see section \ref{ChowlaRH} for more details). Therefore, according to our results and analysis, we obtain that  Chowla 
conjecture implies Riemann Hypothesis. The proof of this last fact is essentially based on the computation of the $L^\alpha$ norms of 
the $L^2$-normalized polynomials generated  by Liouville or M\"{o}bius function.\\

Our result on Chowla conjecture and RH confirm in some sense the heuristic argument of Denjoy \cite{Donjoy}, that is, RH 
``holds with probability one'' (for more details see section \ref{ChowlaRH}).\\


Let us remind that P. Erd\"{o}s wrote about his ultraflat conjecture in \cite{ErdosIV} the following:\\

`` Some of these  questions  may not be "serious"  Mathematics
but  I am sure  the  following  final  problem  considered  by D. J. Newman and myself  for  a long  time  is both  difficult  
and interesting:  Let  $\epsilon_k=\pm 1.$  Is it  true  that  there  is  an absolute  constant  $c$  so that  for  
every  choice  of the  $\epsilon_k's$
$$\max_{|z|=1}\Big|\sum_{k=1}^{n}\epsilon_kz^k\Big| > (1+c)\sqrt{n}?$$
This  probably  remains  true  if  the  condition  $\epsilon_k=\pm 1$ is
replaced  by  $|\epsilon_k|= 1$.  ''\\

The last conjecture (when  $\epsilon_k=\pm 1$ is
replaced  by  $|\epsilon_k|= 1$) was disproved by J-P. Kahane  \cite{Kahane}. Further, J. Beck {\cite{Beck}} has shown that 
the sequence 
\begin{eqnarray}\label{type2}
K_j(z)\egdef\frac1{\sqrt{q_j}}\sum_{k=0}^{q_j-1}a_{k,j} z^k,
\end{eqnarray}
$j=1,\cdots$ can be chosen to be flat in the sense of Littlewood (see definition below) with  coefficients chosen from the solutions of 
$z^{400} =1$. So Beck's result naturally raises the question as to what is the smallest cyclic subgroup of the circle group which supplies 
coefficients $a_{k,j}$'s for a flat sequence of trigonometric polynomials of the type \eqref{type2} (flat in the sense of Littlewood). 
According to the extensive numerical computations described in \cite{odlyzko}, the flat polynomials in the sense of Littlewood 
of the type \eqref{Littlewood}  may exist  \cite{Brad}. However, 
our arguments seems to be far form bring any contribution to this problem. 
Although, our result say that there are no ultraflat polynomials of the type \eqref{Littlewood} and this 
confirm the numerical evidence in \cite{odlyzko}.\\

 Let us further mention that several attempts have been made in the past  to solve Turyn-Golay's conjecture or Erd\"{o}s's conjectures. 
 For a brief review on those attempts we refer the reader to \cite{Jedwab92} and \cite{Saffari82}.\\

 It turns out that Turyn-Golay's conjecture is related to some spectral problems in ergodic theory. Indeed, 
 T. Downarowicz and Y. Lacroix established that Turyn-Golay's conjecture is true 
 if and only if all binary Morse flows have singular spectrum \cite{Down}. As a consequence to our main result, we obtain that 
 all binary Morse flows have singular spectrum.\\
 
 It come to us a pleasant surprise that there is also a connection between the pairwise independence sequence 
 and the famous Banach-Rohklin problem on the existence of dynamical system with pure Lebesgue spectrum and finite multiplicity. 
 This connection was made by Robertson \cite{R1988} and Womack \cite{W1984}. There is many investigations 
 in this direction on the Banach-Rohklin problem. Unfortunately, to the best of author's knowledge, 
 none of those investigations so far was successful. For more details, we refer to \cite{klopo} and the references therein.
 We also refer to \cite{Etemadi} and \cite{Bradely2009} for the more recent results on the subject.\\
 
 It turn out also that the study of Banach-Rohklin problem 
 in some class of dynamical system is equivalent to the $L^1$ flatness problem in the class of Littlewood polynomials. 
 This later connection was made by M. Guenais in 
 \cite{Guenais}. For more details we refer to section \ref{spectral}.\\ 
 

 The paper is organized as follows. In section \ref{intro}, we set up notation and terminology, and we further review some fundamental tools 
 from the interpolation theory in Fourier analysis. In section \ref{main-R1},  we introduce the notion of flat polynomials and 
 we state our first main result. In section \ref{spectral}, we remind the notion of Barker sequence and we review some ingredients from the spectral 
 theory of dynamical systems. We further state our second and third main results. In section \ref{proof1}, we present the proof of 
 our first main result. In section \ref{dynamical}, we present a dynamical proof of our main first main result and the proof our second and 
 third main results. In section \ref{ChowlaRH}, we provide an application of our dynamical proof to study the flatness issue under 
 Chowla conjecture and Riemann hypothesis. As a consequence, we obtain that Chowla conjecture implies Riemann hypothesis.
 Finally, in section \ref{Remarks}, we state some remarks and open questions.

\section{Notation, definitions and tools.}\label{intro}

Let $\cl$ denote the class of Littlewood polynomials by which we mean the trigonometric polynomials of the type
  \begin{eqnarray}\label{Littlewood2}
   P(z) = \frac{1}{\sqrt q}( \epsilon_0 + \epsilon_1 z +\cdots + \epsilon_{q-2} + \epsilon_{q-1} z^{q-1}),~~z \in S^1,
   \end{eqnarray}
where $\epsilon_0 = \epsilon_{q-1} =1,\epsilon_i =\pm 1 $, $1\leq i\leq q-2$.   Note that we make first and last coefficient of $P$ 
positive in our definition. This makes  the correspondence $T$ defined below one-one. Let $\nB$ denote  the class of  Newman-Bourgain 
polynomials, i.e., polynomials $\tilde{P}$ of type
$$\frac{1}{\sqrt m}(\eta_0 + \eta_1z + \cdots +\eta_{q-2}z^{q-2} + \eta_{q-1}z^{q-1}),$$
where $\eta_0 = \eta_{q-1} =1$, $\eta_i = 0$ or $1, 1\leq i \leq q-2$, and where $m$ is the number of non-zero terms in $\tilde{P}$ which is also the number of $i$ with $\eta_i =1$. Note that if $P$ is as in \eqref{Littlewood} and if we put
$$\eta_i = \frac{1}{2}(\epsilon_i +1), 0\leq i\leq q -1$$
then the polynomial
$$\frac{1}{\sqrt m}(\eta_0 + \eta_1z + \cdots + \eta_{q-2}z^{q-2} + \eta_{q-1}z^{q-1})$$
is in class $\nB$, where $m$ is the number of $\eta_i = 1$  which is also the number of $\epsilon_i =1$. \\

\noindent{}It is obvious that the finite sequence $(\eta_j)_{j=0}^{q-1}$ defined a subset of the set $\big\{0,\cdots,q-1\big\}$ 
which we denote by $H$. By abuse of notation we will denote also by $H$ the subset associated to the sequence $(\eta_j)_{j=0}^{+\infty}$.\\

As is customary, we denote by
$\xi_{q,j}$, $j=0, \cdots,q-1$ the $q$-root of unity given by
$$\xi_{q,j}=e^{2\pi i\frac{j}{q}}.$$
and by $\# A$ the cardinal of the set $A$. For $r \geq 2$, the discrete Fourier transform of the finite set $A$ mod $r$ and its 
balanced function are given by

$$ DF_r{(\1_A)}(\ell)=\frac{1}{r}\sum_{j=0}^{r-1}A(j) \xi_{r,j \ell},~~~~~~~~~\ell=0,\cdots,r-1,$$
where $A(j)=\#\big\{k \in A~~~:~~~k \equiv j\big\}$ or $A(j)-\#A$ for its balanced function. Notice that in the usual sense
$DF_r{(\1_A)}$ is the discrete Fourier transform  of the function $f$ given by $f(j)=A(j),~~~~j=0,\cdots,r-1.$\\

The class of polynomials with coefficients of modulus one is denoted by $\mathcal{G}$. To avoid heavy notation, we denoted also by $\mathcal{G}$
the class of $L^2$-normalized polynomials from class  $\mathcal{G}$.
\vskip 0.1cm

\paragraph{\textbf{A formula between Littlewood and Newman-Bourgain Polynomials.}}
\vskip 0.1cm
Let us define one-one invertible map $ T$ from the class $L$ to the class $\nB$ by
\begin{eqnarray*}
(T(P))(z) &=& T\Big(\frac{1}{\sqrt q}\Big(\epsilon_0 + \epsilon_1z + \cdots + \epsilon_{q-2}z^{q-2} + \epsilon_{q-1}z^{q-1}\Big)\Big)\\
&=& \frac{1}{\sqrt m}\Big(\eta_0 + \eta_1z + \cdots + \eta_{q-2}z^{q-2} + \eta_{q-1}z^{q-1}\Big),
\end{eqnarray*}
where $\eta_i = \frac{1}{2}(\epsilon_i +1), 0\leq i\leq q -1$, and $m$ is the  number of $\eta_i =1$ which is
 also the number of $\epsilon_i =1$.\\

Note
$$T^{-1}\Big(\frac{1}{\sqrt m}\Big(\sum_{i=0}^{q-1}\eta_iz^i\Big)\Big) =
\frac{1}{\sqrt q}\Big(\sum_{i=0}^{q-1}(2\eta_i -1)z^i\Big). $$

Let
$$D(z) = D_q(z) = \frac{1}{\sqrt q}\sum_{i=0}^{q-1} z^i .$$
We thus have that $D(1) = \sqrt q$, while for $z \in S^1\setminus\{1\}$,
$$D(z) = \frac{1}{\sqrt q}\frac{1-z^q}{1-z} \rightarrow 0$$
as $q \rightarrow \infty$.\\

The formula for polynomials in $\cl$ mentioned in the title of this subsection is as follows: If $P$ is as in \eqref{Littlewood2}
then

\begin{eqnarray}\label{Littlewood3}
P(z) &=& 2\frac{\sqrt m}{\sqrt q}(T(P))(z) - D(z), \nonumber\\
       &=&2\frac{1}{\sqrt q} A(z) - D(z) \nonumber\\
       &=& Q(z)-D(z).
\end{eqnarray}

\noindent{}where $m$ is the number of terms in $P$ with coefficient +1, $A(z) = \sqrt m~~ T(P)(z)$, and
 $$Q(z)=Q_q(z)=\frac{2}{\sqrt{q}}A(z).$$
 The proof follows as soon as we write $T(P)(z)$ and $D(z)$ in the right hand side in full form and collect 
 the coefficient of $z^i, ~~0\leq i \leq q-1$.\\

\noindent{} Note that polynomials in $\cl$, $\nB$ and the polynomial $D$ all have $L^2(S^1, dz)$ norm 1. Moreover, for each $j \in \{1,\cdots,q-1\}$,
we have
\begin{eqnarray}\label{rootof1}
P(\xi_{q,j})=Q(\xi_{q,j}),~~~~~~~~~~~~j=1,\cdots,q-1,
\end{eqnarray}\\
by the identity \eqref{Littlewood3}.
\vskip0.2cm
We will further need the following fundamental inequalities from the interpolation theory due to Marcinkiewz \& Zygmund \cite[Theorem 7.10, Chapter X, p.30]{Zygmund}.
\begin{Th}\label{MZ}
For
$\alpha > 1$, $n \geq 1$, and any analytic trigonometric polynomial $P$ of degree $\leq n$,
\begin{eqnarray}\label{MZ1}
\frac{A_{\alpha}}{n}\sum_{j=0}^{n-1}\big|P(\xi_{n,j})\big|^{\alpha}
\leq \int_{S^1}\Big|P(z)\Big|^{\alpha} dz \leq \frac{B_{\alpha}}{n}\sum_{j=0}^{n-1}\big|P(\xi_{n,j})\big|^{\alpha},
\end{eqnarray}
where  $A_{\alpha}$ and $B_{\alpha}$ are independent of $n$ and $P$.
\end{Th}
For the trigonometric polynomials, Marcinkiewz-Zygmund interpolation inequalities can be stated as follows \cite[Theorem 7.5, Chapter X, p.28]{Zygmund}.
\begin{Th}\label{MZ2}
For $\alpha > 1$, $n \geq 1$, and any trigonometric polynomial $P$ of degree $\leq n$,
\begin{eqnarray}\label{MZ3}
\quad\quad\quad\frac{A_{\alpha}}{2n+1}\sum_{j=0}^{2n}\big|P(\xi_{2n+1,j})\big|^{\alpha}
\leq \int_{S^1}\Big|P(z)\Big|^{\alpha} dz \leq \frac{B_{\alpha}}{2n+1}\sum_{j=0}^{2n}\big|P(\xi_{2n+1,j})\big|^{\alpha},
\end{eqnarray}
where  $A_{\alpha}$ and $B_{\alpha}$ are independent of $n$ and $P$.
\end{Th}
We will also need the following special case of Marcinkiewicz-Zygmund inequalities for the $L^2$ and $L^4$ norms. 
We include the proof for sake of completeness.
\begin{lem}\label{classical} Let $P$ be an analytic trigonometric polynomial with degree  $ \leq q-1$ and complex coefficients 
$a_j,~~~j=0,\cdots q-1$. Then
\begin{eqnarray}\label{obs1}
\frac1{q}\sum_{j=0}^{q-1}\big|P(\xi_{q,j})\big|^2=\int \big|P(z)\big|^2 dz.
\end{eqnarray}
Assume further that the degree of $P$ is odd and the coefficients are real, then
\begin{eqnarray}\label{obs2}
\frac{1}{2q}\sum_{j=0}^{q-1}\big|P(\xi_{q,j})\big|^4 +\frac{1}{2q}\sum_{j=0}^{q-1}\big|P(-\xi_{q,j})\big|^4=
\int \big|P(z)\big|^4 dz.
\end{eqnarray}
\end{lem}
\begin{proof} An easy computation gives, for any $z \in S^1$,
\begin{eqnarray}\label{correlation1}
\big|P(z)\big|^2&=&\sum_{k,l=0}^{q-1}a_k \overline{a_l}z^{k-l} \nonumber \\
&=&\sum_{k=0}^{q-1}|a_k|^2+\sum_{\overset{k,l=0}{k \neq l}}^{q-1}a_k \overline{a_l}z^{k-l},
\end{eqnarray}
We thus get
\begin{eqnarray*}
\frac1{q}\sum_{j=0}^{q-1}\big|P(\xi_{q,j})\big|^2 &=&
\sum_{k=0}^{q-1}|a_k|^2+\sum_{\overset{k,l=0}{k \neq l}}^{q-1}a_k \overline{a_l}\Big(\frac1{q}\sum_{j=0}^{q-1} \xi_{q,j}^{k-l}\Big)\\
&=&\sum_{k=0}^{q-1}|a_k|^2,
\end{eqnarray*}
since
$$\sum_{j=0}^{q-1} \xi_{q,j}^{l}=0,\textrm{~~for~~all~~} l \not \equiv 0 ~~~~{\textrm {mod~~}}q.$$
The identity \eqref{obs1} follows once we observe that
$$\int_{S^1} \big|P(z)\big|^2 dz=\sum_{k=0}^{q-1}|a_k|^2.$$
For the proof of \eqref{obs2}, we rewrite \eqref{correlation1} as follows
\begin{eqnarray}\label{correlation2}
 \big|P(z)\big|^2= c_0+\sum_{\overset{l \neq 0}{|l|\leq q-1}}c_lz^l,
\end{eqnarray}
where $(c_l)$ are the autocorrelations of the sequence $(a_k)_{k=0}^{q-1}$. Since the coefficients are real, \eqref{correlation2} take the 
following form.
\begin{eqnarray*}
 \big|P(z)\big|^2= c_0+\sum_{1 \leq l \leq q-1}c_l\big (z^l+z^{-l}\big),
\end{eqnarray*}
Hence
\begin{eqnarray}\label{key-4}
 \big|P(z)\big|^4= c_0^2&+&2c_0\sum_{1 \leq l \leq q-1}c_l\big(z^l+z^{-l}\big) \nonumber \\&+&
 \sum_{k,l}c_kc_l(z^{l+k}+z^{-(l+k)}\big)
 +\sum_{l \neq k}c_lc_k (z^{l-k}+z^{k-l}\big)+2\sum_{k=1}^{q-1}c_k^2.
\end{eqnarray}
Proceeding in the same manner as before we obtain
\begin{eqnarray}\label{four1}
 \frac1{2q}\sum_{j=0}^{q-1}\big|P(\xi_{q,j})\big|^4=\frac12\Big(c_0^2+2\sum_{k=1}^{q-1}c_kc_{q-k}+2\sum_{k=1}^{q-1}c_k^2\Big),
\end{eqnarray}
and
\begin{eqnarray}\label{four2}
 \frac1{2q}\sum_{j=0}^{q-1}\big|P(-\xi_{q,j})\big|^4=\frac12\Big(c_0^2+2(-1)^q\sum_{k=1}^{q-1}c_kc_{q-k}+2\sum_{k=1}^{q-1}c_k^2\Big).
\end{eqnarray}
Since $q$ is odd, by  adding \eqref{four1} to \eqref{four2} we get \eqref{obs2}. This finishes the proof of the lemma.
\end{proof}
\begin{rem}(1) The key argument in the proof of Lemma \ref{classical} can be reformulated as follows: let
$$\mu_q=\frac1{q}\sum_{j=0}^{q-1}\delta_{\xi_{q,j}},$$
and
$$e_l(z)=z^l, \textrm{~~for~~all~~} l \in \Z.$$
Then, for all $k,l \in \Z$, we have
\begin{eqnarray*}
<e_l,e_k>&=& \int_{S^1}  e_l(z) \overline{e_k(z)} d\mu_q\\
&=& \left\{
      \begin{array}{ll}
        0, & \hbox{if $l \not \equiv k$ mod $q$;} \\
        1, & \hbox{if not,}
      \end{array}
    \right.
\end{eqnarray*}
that is, the family $\{e_l\}_{l=0}^{q-1}$ is an orthonormal basis for $L^2(\mu_q)$. We further notice that
$\{e_l\}_{l \geq 0}$ is an orthonormal basis for $L^2(dz)$. Moreover, it is well known that 
the previous arguments are at the heart of the Fast Fourier transform
(FFT) techniques, we refer to \cite[Chap.7]{Stein-Sh} for more details. For the application of FFT in a variety of areas including
biomedical engineering and the radar communications field, please refer to \cite[Chap. 8]{Kim} and the references given there.
\vskip 0.1cm
(2) Let us further point out that the identities \eqref{four1} and \eqref{four2} can be obtained as an easy application of Parseval identity. Indeed, write
\begin{eqnarray*}
\big|P(\xi_{q,l})\big|^2&=&c_0+\sum_{k=1}^{q-1}c_k e^{2\pi i k\frac{l}{q}}+\sum_{k=1}^{q-1}c_{k} e^{-2\pi i k\frac{l}{q}}\\
&=& c_0+\sum_{k=1}^{q-1} c_k e^{2\pi i k\frac{l}{q}}+\sum_{k=1}^{q-1}c_{q-k} e^{2\pi i k\frac{l}{q}}\\
&=&c_0+\sum_{k=1}^{q-1} (c_k+c_{q-k}) e^{2\pi i k\frac{l}{q}},
\end{eqnarray*}
and apply the Parseval identity to obtain
$$\frac{1}{2q}\sum_{j=0}^{q-1}\big|P(\xi_{q,j})\big|^4=\frac12\Big(c_0^2+\sum_{k=1}^{q-1} (c_k+c_{q-k})^2\Big).$$
An elementary calculation show  that this last identity is exactly the identity \eqref{four1}. For the second
identity, the detailed verification is left to the reader.
\end{rem}
\section{Our first main result and flat polynomials}\label{main-R1}
For any $\alpha>0$ or $\alpha=+\infty$, the sequence $P_n(z), n=1,2,\cdots$ of analytic trigonometric polynomials of $L^2(S^1,dz)$
norm 1 is said to be $L^\alpha$-flat if  $| P_n(z)|, n=1,2,\cdots$ converges in $L^\alpha$-norm to the constant function $1$. For $\alpha=0$, we say that $(P_n)$ is $L^\alpha$-flat, if the Mahler measures $M(P_n)$ converge to 1. 
We recall that the Mahler measure of a function $f \in L^1(S^1,dz)$ is defined by
$$ M(f)=\|f\|_0=\lim_{\beta \longrightarrow 0}\|f\|_{\beta}=\exp\Big(\int_{S^1} \log(|f(t)|) dt\Big).$$

The sequence $P_n(z), n=1,2,\cdots$  is said to be square $L^\alpha$-flat if  $| P_n(z)|^2, n=1,2,\cdots$ converges in $L^\alpha$-norm to the constant function $1$.\\

Obviously, if the sequence $P_n(z), n=1,2,\cdots$ is square $L^\alpha$-flat then it is $L^\alpha$-flat, since
\begin{eqnarray*}
\Big|\big|P_n(z)\big|^2-1\Big|^\alpha&=&\Big|\big|P_n(z)\big|-1\Big|^\alpha \Big(\big|P_n(z)\big|+1\Big)^\alpha\\
&\geq& \Big|\big|P_n(z)\big|-1\Big|^\alpha,~~~~~\rm{if~~} \alpha>0,
\end{eqnarray*}
and
$$M(|P_n|^2)= M(|P_n|)^2.$$

 We further have that the square $L^1$-flatness is equivalent to the $L^1$-flatness, by Proposition 4.2 from \cite{elabdal-Banach}.\\

We say that the sequence $P_n(z), n=1,2,\cdots$  is flat in almost everywhere sense (a.e. $(dz)$) if $| P_n(z)|$, $n=1,2,\cdots$ converges almost everywhere to $1$ with respect to $dz$.\\


Following \cite{elabdal-little}, the sequence $P_n, n=1,2,\cdots$ of polynomials from the class $\mathcal{L}$ (or $\mathcal{G}$) is flat in the sense of Littlewood  if there exist constants $0<A<B$  such that for all $z\in S^1$, for all $n \in \N$ (or at least for a large $n$)

$$A \leq \big| P_n(z)\big| \leq B.$$

It is immediate that the flatness properties are invariant under $S$. It is further a nice exercise that the $L^4$ conjecture of Erd\"{o}s 
and the ultraflat conjecture of Erd\"{o}s holds in the class of Newman-Bourgain polynomials \cite{abd-nad2}, \cite{Abd-Nad}.\\

We will need the following fundamental  criterion of the square $L^2$-flatness.

\begin{Prop}\label{Cri1}Let $(P_q(z))_{q \geq 0}$ be a sequence of $L^2$-normalized polynomials. Then,
$(P_q(z))_{q \geq 0}$  is square $L^2$-flat if and only if the $L^4$-norm of $P_q$ converge to $1$ as $q\longrightarrow +\infty$.
\end{Prop}
\begin{proof} A straightforward computation gives
\begin{eqnarray}\label{identity2}
\int_{S^1} \Big|\big|P_q(z)\big|^2-1\Big|^2 dz&=& \int_{S^1} |P_q(z)|^4 dz-2 \int_{S^1} \big|P_q(z)\big|^2 dz +1 \nonumber\\
&=&\int_{S^1} |P_q(z)|^4 dz-1.
\end{eqnarray}
Therefore,
$$\int_{S^1} \Big|\big|P_q(z)\big|^2-1\Big|^2 dz \tend{q}{+\infty}0,$$
if and only if
$$\int_{S^1} |P_q(z)|^4 dz \tend{q}{+\infty}1.$$
The proof of the proposition is complete.
\end{proof}
We can strengthen Proposition \ref{Cri1} as follows
\begin{Prop}\label{Cri2}Let $(P_q(z))_{q \geq 0}$ be a sequence of $L^2$-normalized polynomials. Then, for any integer $p \geq 1$,
the $L^{2p}$-norm of $P_q$ converge to $1$ as $q\longrightarrow +\infty$ if and only if
$$ \int \Big|\big|P_q(z)\big|^2-1\Big|^p dz \tend{q}{+\infty}0.$$
\end{Prop}
\begin{proof}By the triangle inequality, we have 
$$\Big\|\big|P_q\big|^2-1\Big\|_p \geq \Big| \big\|P_q\|_p-1\Big|.$$
Whence $\Big| \big\|P_q\|_p-1\Big|$ converge to $0$ as $q \longrightarrow +\infty$ if $\Big\|\big|P_q\big|^2-1\Big\|_p$ converge to zero 
as $q \longrightarrow +\infty$. For the opposite direction, since $p$ 
is an integer, we can thus write
$$\Big\|\big|P_q\big|^2-1\Big\|_p^p=\sum_{k=0}^{p}\binom{p}{k} (-1)^k\int \big|P_q(z)\big|^{2k} dz,$$
We further have 
$$1=\Big\|P_q\|_2 \leq \Big\|P_q\Big\|_{2k} \leq \Big\|P_q\Big\|_{2p},$$
for any $k=1,\cdots,p$, by H\"{o}lder inequality. Assume that $\Big\|P_q\Big\|_{2p}$ converge to $1$ as $q \longrightarrow +\infty$ it follows that 
$$ \Big\|P_q\Big\|_{2k} \tend{q}{+\infty}1,\quad \quad \textrm{for}\quad 1 \leq k \leq p.$$
We thus get 
$$\sum_{k=0}^{p}\binom{p}{k} (-1)^k\int \big|P_q(z)\big|^{2k} dz \tend{q}{+\infty}(1-1)^p=0.$$
This finish the proof of the proposition.
\end{proof}

\vskip 0.5cm

We are now able to state our first main result.
\begin{Th}\label{mainL1} There are no square $L^2$-flat polynomials sequence from the class of Littlewood polynomials.
\end{Th}
From this, it follows that the $L^4$-conjecture and the ultraflat  conjecture of Erd\"{o}s holds for the class of Littlewood polynomials.\\


\section{Our second and third main results.}\label{spectral}

Before stating our second and third main results, we need to recall the notion of Barker sequences and some basic facts on the notion of dynamical systems arising from generalized Morse sequences.\\

\paragraph{\textbf{Baker sequences and the connection to digital communications engineering.}} Barker sequences are well-known in the streams of investigation from digital communications
engineering. Barker introduced such sequences in \cite{Barker}  to produce a low autocorrelation binary sequences, or equivalently  a binary sequence with the highest possible value of $F$. The largest well-known values of $F$ are $F_{12}=14.0833$ and $F_{10}=12.1$  obtain respectively by the following sequences $$1,-1,1,-1,1,1,-1,-1,1,1,1,1,1,$$   and
$$1,-1,1,1,-1,1,1,1,-1,-1,-1.$$
\noindent{}No other merit factor exceeding $10$ is known for any $n$. It was conjectured that $169/12$ and $121/10$ are the
maximum possible values for $F$. This conjecture still open.\\

Given a binary sequence $\mathbf{b}=(b_j)_{j=0}^{n-1}$, that is, for each $j=0,\cdots,n-1$, $b_j=\pm 1$. The $k$-th aperiodic autocorrelation of $\mathbf{b}$ is given by
$$c_k=\sum_{j=0}^{n-k-1}b_j b_{j+k},~~~~{\textrm{~~for~~}} 0 \leq k \leq n-1.$$

For $k < 0$ we put $c_k = c_{-k}$. $\mathbf{b}$ is said to be a Barker sequence if for each $k \in \big\{1,\cdots,n-1\big\}$ we have
$$|c_k| \leq 1, \textrm{~~that is,~~}c_k=0,\pm 1.$$

The Barker sequences and their generalizations have been a subject of many investigations since 1953, both from digital communications engineering view point and complex analysis viewpoint. Therefore, there is an abundant literature on the subject, we refer to \cite{Saffari-Barker}, \cite{Jedwab}, \cite{Borwein2}, \cite{Borwein3}, \cite{Mossinghoff} and the references therein for more details. Here, we remind only the following result need it.

\begin{Th}\label{Barker} Let $(b_i)_{i=1}^{n}$ be a Barker sequence with length $n$.
\begin{enumerate}
\item If $n$ is odd then $n \leq 13$, if not and $n>2$ then $n=4m^2$ for some integer $m$.
\item Assume further that there exist a Barker sequence with arbitrary length and let $P_n$ be a Littlewood polynomial whose coefficients form a Barker sequence of length $n$. Then the sequence $(P_n)$ is square $L^2$-flat.
\end{enumerate}
\end{Th}
\begin{proof}(1) is due to Turyn and  Storer \cite{TS}. The second part (2) is essentially due to Saffari \cite{Saffari-Barker}, we refer also to the proof of Theorem 4.1 in \cite{Borwein2} line 5.

\end{proof}
\vskip 0.1cm
At this point, let us state our second main result.

\begin{Th}\label{mainL2} There are only finitely many Barker sequences.
\end{Th}

\paragraph{\textbf{Weak Banach-Rohklin spectral problem and the connection with Ergodic Theory.}}In \cite{Bourgain} Bourgain showed  
that the $L^1$-flat polynomials problem in the class $\mathcal{NB}$ is related to the nature of the spectrum of 
the Class 1 maps introduced by Ornstein \cite{Orn}. This class of maps is nowadays called rank one maps and there is a large literature on it. Later, M. Guenais in  \cite{Guenais} established that  $L^1$-flat polynomials problem in the class $\mathcal{L}$ is equivalent to 
the weak Banach-Rohklin problem in the class of Morse coycle extension maps.\\

Following Ulam the Banach problem from the Scottish Book can be stated  as follows\cite[p.76]{Ulam}.

\begin{que}[Banach Problem]
Does there exist a square integrable function $f(x)$ and a measure preserving transformation $T(x)$,
$-\infty<x<\infty$, such that the sequence of functions $\{f(T^n(x)); n=1,2,3,\cdots\}$ forms a complete
orthogonal set in Hilbert space?
\end{que}
\vskip 0.1cm
Obviously, Banach problem has a positive answer in the class of non-conservative dynamical systems. For the conservative case, 
 the problem remained open until very recently, when it was answered affirmatively by the author in \cite{elabdal-Banach}. Precisely, therein,
 the author produced  a conservative ergodic 
infinite measure preserving with simple Lebesgue spectrum. For more details, we refer to \cite{elabdal-Banach}.\\

The Russian related problem to the Banach problem is known as Rohklin problem. Rohklin  asked in \cite{Rokhlin} on finding a map acting on 
the finite measure space with finite Lebesgue spectrum.  To the best knowledge of the author, this problem still open. However, 
the weak Rohklin problem on the existence of dynamical system with finite Lebesgue component was solved since 1982 
by J. Mathew and M. G. Nadkarni \cite{MN}, T. Kamae \cite{Kamae}, M. Queffelec \cite{Queffelec}, and O. Ageev \cite{Ag}. Indeed, 
the authors produced a dynamical system with Lebesgue component of multiplicity two. Fifteen years later, 
M. Guenais produced a torsion group action with Lebesgue component of multiplicity one \cite{Guenais}.
\vskip 0.1cm
  Historically, the problem on finding a map acting on a probability space with simple Lebesgue spectrum seems 
 to be initiated by Banach and Rohklin. On this problem, and more generally, on Rohklin problem,  
 Kirillov in his survey paper \cite{Kiri} wrote ``there are grounds for thinking that such examples do not exist". 
 The weak Banach-Rohklin problem raised the question of whether there exists an ergodic map acting on a probability space 
 with simple Lebesgue component plus some singular part in the spectrum. If we require only that the map is non-singular and ergodic then, 
 very recently, M. Nadkarni and the author established that the problem has an affirmative answer \cite{AbN1}.
 \vskip 0.1cm
 Here, we will summarize briefly the connection between the square $L^2$-flatness and the so-called Morse dynamical systems.
 \vskip 0.1cm

 In the symbolic dynamics language, let $B_k$ be a block of length $k$ in the alphabet $\mathfrak{A}=\{-1,+1\}$, that is,
 $B_k=(b_0,\cdots,b_{k-1}),$ $b_i \in \mathfrak{A}$. The $b_i$ are also denoted by $B_k[i]$. If $C_k$ and $D_m$ are two blocks then the concatenation operation $(CD)_{k+m}$ and the product operation of $(C \times D)_{k.m}$ is defined respectively by
 $$(CD)_{k+m}=C_k[0] \cdots C_k[k-1]D_m[0]\cdots D_m[m-1],$$
 and
 $$(C \times D)_{k.m}[s+t.k]=C_k[s]D_m[t],~~~~~~~~~s=0,\cdots,k-1,~~~~~t=0,\cdots,m-1.$$

 To each binary sequence $B \in \big\{+1,-1\big\}$ we associate the Littlewood trigonometric polynomial defined by
  $$P_{B,k}(z)=\frac1{\sqrt{k}}\sum_{j=0}^{k-1}B_k[j]z^j,~~~~~z \in S^1.$$

Given a sequence of blocks $B_{k_1},B_{k_2},\cdots$ satisfying
\begin{eqnarray}\label{CVD}
B_{n_p}[0]=1, ~~~~~~~~\forall p \in \N.
\end{eqnarray}

 The one-sided generalized  Morse sequence $A$ is defined as the coordinatewise limit of the blocks $A_{n_1.\cdots n_p}=B_{n_1} \times B_{n_2}\times \cdots \times B_{n_p}.$ Notice that convergence  is granted by the condition \eqref{CVD}.\\

 The one-sided generalized  Morse sequence $A$ is extended in the usual manner to be the bi-infinite sequence which we still denoted by $A$.\\
\vskip 0.1cm
 Let $S$ be a shift map on the compact space $\{+1,-1\}^{\Z}$ given by $S(x)[n]=x[n+1]$, $n \in \Z$, and $A \in \{+1,-1\}^{\Z}$,
be a generalized Morse sequence. The Morse flow is defined as the subshift generated by $A$, that is, the topological dynamical
system $(X_A,S)$ where $X_A$ is the closure of the orbit of $A$ under $S$. \\

In the language of cocycles, the Morse flow can be defined as $2$-point extension of the odometer. The associated cocycle $\phi$  
is defined inductively, and it is continuous at all points except at one point. $\phi$ is also called a continuous Morse cocycle.\\

The generalized Morse sequences and Morse cocycle has been the subject of many investigations and publication, for more details we refer to \cite{Kdown}, \cite{Guenais} and the references therein.\\

We remind that  M. Guenais in \cite{Guenais} established that the Morse cocycle has a simple Lebesgue spectrum if and only if there exists a sequence of $L^1$-flat polynomials.\\
\vskip 0.1cm
For a continuous cocycle, Downarowicz-Lacroix in \cite{Down} proved the following:
\begin{Th}\label{dow-L} The dynamical system arising from the continuous Morse sequence $A=(A_n)$ has a simple Lebesgue component if and only if the polynomials $(P_{A,n}(z))$ are square $L^2$-flat.
\end{Th}
\vskip 0.2cm

Our third main result concern the spectrum of dynamical system arising from the generalized Morse sequences, and it can be stated as follows
\begin{Th}\label{mainL3}
The spectrum of any Morse flow arising from continuous Morse sequence is singular.
\end{Th}
\section{Proof of the main results}\label{proof1}
We start by recalling the following special case of the $L^\alpha$-flatness criterion from \cite{elabdal-Banach}. 
We include the proof for sake of completeness.
\begin{Prop}Let $\alpha>1$ and $(P_q(z))_{q \geq 0}$ be a sequence of $L^2$-normalized polynomials such 
that for each $q$ the degree of $P_q$ is $q-1$. Then $P_q,~~~~q=1,\cdots,$ are square $L^\alpha$-flat if and only if
$$\frac1{q}\sum_{j=0}^{q-1}\Big|\big|P_q(\xi_{q,j})\big|^2-1\Big|^\alpha \tend{q}{+\infty}0,$$
and
$$\frac1{q}\sum_{j=0}^{q-1}\Big|\big|P_q(\xi_{2q,2j+1})\big|^2-1\Big|^\alpha \tend{q}{+\infty}0.$$
\end{Prop}
\begin{proof}This is an easy application of Marcinkiewicz-Zygmund inequalities combined with 
the following obvious observation:
$$\xi_{2q,2j}=\xi_{q,j},~~~~~~~~~~~~~~~j=0,\cdots,q-1.$$
\end{proof}
As an easy consequence we have the following corollary.
\begin{Cor}\label{appMZ} 
 Let $(P_q(z))_{q \geq 0}$ be a sequence of $L^2$-normalized polynomials such 
that for each $q$ the degree of $P_q$ is $q-1$. Then  $P_q,~~~~~q=1,\cdots,$ are square $L^2$-flat if and only if
$$\frac1{q}\sum_{j=0}^{q-1}\Big|\big|P_q(\xi_{q,j})\big|^2-1\Big|^2 \tend{q}{+\infty}0,$$
and
$$\frac1{q}\sum_{j=0}^{q-1}\Big|\big|P_q(\xi_{2q,2j+1})\big|^2-1\Big|^2 \tend{q}{+\infty}0.$$
\end{Cor}

We further need the following proposition due to Jensen-Jensen and H{\o}holdt  \cite{Hoholdt-JJ} (see also \cite{elabdal-little}). 
We will present a simple proof of it in section \ref{dynamical}.
\begin{Prop}
\label{JJH}Let $(P_q(z))_{q \geq 0}$ be a sequence of  $L^2$-normalized Littlewood polynomials. Suppose that
$$\frac{\#\big\{~~j~~: \epsilon_j=-1\big\}}{q} \longrightarrow \fr(-1)$$ as $q\longrightarrow +\infty$. If $\fr(-1) \neq \frac12$ then
$$\Big\|P_q\Big\|_4 \tend{q}{+\infty}+\infty.$$
\end{Prop}
We are now able to prove our first main result (Theorem \ref{mainL1}).
\begin{proof}[\textbf{Proof of Theorem \ref{mainL1}.}]
Assume that there exists a sequence $(P_q)$ of $L^2$-normalized Littlewood polynomials such that the $L^4$-norm $(P_q)$ converge to 1 
as $q\longrightarrow+\infty$. Furthermore, without loss of generality, assume that the degree of $(P_q)$ are odd. Then, 
by appealing to Proposition \ref{Cri1}, the sequence of the polynomials $(P_q)$ is square $L^2$-flat. We 
thus have
$$\frac1{q}\sum_{j=0}^{q-1}\Big|\big|P_q(\xi_{q,j})\big|^2-1\Big|^2 \tend{q}{+\infty}0,$$
by Corollary \ref{appMZ}. Combining this with \eqref{rootof1}, we obtain
 $$\frac1{q}\sum_{j=1}^{q-1}\Big|\big|Q_q(\xi_{q,j})\big|^2-1\Big|^2 \tend{q}{+\infty}0.$$
Consequently,
\begin{eqnarray}\label{id1}
\frac1{q}\sum_{j=1}^{q-1}\big|Q_q(\xi_{q,j})\big|^4-\frac2{q}\sum_{j=1}^{q-1}\big|Q_q(\xi_{q,j})\big|^2+
\frac{q-1}{q} \tend{q}{+\infty}0.
\end{eqnarray}
Now, applying Lemma \ref{classical},  we get
\begin{eqnarray}\label{id2}
\quad \quad\frac1{q}\sum_{j=1}^{q-1}\Big|\big|Q_q(\xi_{q,j})\big|^2-1\Big|^2 &=&
\frac1{q}\sum_{j=0}^{q-1}\big|Q_q(\xi_{q,j})\big|^4-\frac{\big|Q_q(1)\big|^4}{q}-2 \int_{S^1}|Q_q|^2 dz+\\ && 
2\frac{|Q_q(1)\big|^2}{q}+\frac{q-1}{q}
\tend{q}{+\infty}0, \nonumber
\end{eqnarray}
We further have
$$|Q_q(1)\big|^2=\Big(\frac{2|H|}{\sqrt{q}}\Big)^2=\frac{4|H|^2}{q}, $$
and
$$|Q_q(1)\big|^4=\frac{16|H|^4}{q^2}.$$
Moreover, for any $z \in S^1$,
$$|Q_q(z)\big|^2=\frac{4|H|}{q}+\frac{4}{q}\sum_{l \neq 0 }c_l z^l, $$
where $(c_l)$ are the autocorrelation coefficients of $\{\eta_j\}_{j=0}^{q-1}.$ These autocorrelation coefficients are also 
called {\it{the autocorrelation coefficients of $H$.}}\\

\noindent Integrating, we get
$$\int_{S^1}|Q|^2 dz=\frac{4|H|}{q}.$$
We can thus rewrite \eqref{id2} as follows
\begin{eqnarray*}
\frac1{q}\sum_{j=1}^{q-1}\Big|\big|Q_q(\xi_{q,j})\big|^2-1\Big|^2
= \underbrace{\frac1{q}\sum_{j=0}^{q-1}\big|Q_q(\xi_{q,j})\big|^4-\frac{16\big|H\big|^4}{q^3}-\frac{8|H|}{q}+\frac{8|H|^2}{q^2}+\frac{q-1}{q}}_{\xdownarrow{0.5cm}q \rightarrow +\infty}\\
0 \quad \quad \quad \quad \quad \quad \quad \quad \quad \quad \quad \quad
\end{eqnarray*}
Taking into account our assumption, we see by Proposition \ref{JJH} that
$$-\frac{8|H|}{q}+\frac{8|H|^2}{q^2}+ \frac{q-1}{q} \tend{q}{+\infty}-1.$$
This allows us to assert that the sequence
$ \ds \Big(\frac1{q}\sum_{j=0}^{q-1}\big|Q_q(\xi_{q,j})\big|^4-\frac{16\big|H\big|^4}{q^3}\Big)_{q \geq 1}$ \\converge to $1$ as
$q\longrightarrow +\infty$ .\\

At this point, we claim that for $q$ large enough,
\begin{eqnarray}\label{equiv-D}
\|Q_q\|_4 \sim \|D_q\|_4. 
\end{eqnarray}

To see this, write
\begin{eqnarray*}
\|Q_q\|_4&=&\|Q_q-D_q +D_q\|_4,
\end{eqnarray*}
and by appealing to the triangle inequality, we get
\begin{eqnarray}\label{triangle}
\|D_q\|_4- \|Q_q-D_q \|_4 \leq \|Q_q\|_4 \leq \|Q_q-D_q\|_4 +\|D_q\|_4.
\end{eqnarray}
Moreover, a straightforward computation gives 
\begin{eqnarray}\label{Dirichlet:1}
\big|D_q(z)\big|^2=1+\frac1{q}\sum_{\ell \neq 0} (q-\big|\ell\big|) z^{\ell},
\end{eqnarray}
and
\begin{eqnarray}\label{Dirichlet:2}
\big\|D_q(z)\big\|_4^4&=&1+\frac2{q^2}\sum_{\ell = 1}^{q-1} \ell^2, \nonumber \\
&=&1+\frac2{q^2}.\frac{(q-1)q(2q-1)}6. \nonumber \\
&=&\frac{2}{3}q+\frac{1}{3q}.
\end{eqnarray}
Therefore, for  $q$ large enough,
$$\big\|D_q\big\|_4^4 \sim \frac23.q.$$
Combining this with our assumption, we obtain
$$
\frac{\|Q_q-D_q\|_4 +\|D_q\|_4}{\|D_q\|_4}=\frac{\|Q_q-D_q\|_4}{\|D_q\|_4}+1 \tend{q}{+\infty}1$$
and
$$
\frac{\|D_q\|_4-\|Q_q-D_q\|_4}{\|D_q\|_4}=1-\frac{\|Q_q-D_q\|_4}{\|D_q\|_4}\tend{q}{+\infty}1.$$
Whence, by \eqref{triangle},
$$\frac{\big\|Q\big\|_4}{\big\|D_q\big\|_4} \tend{q}{+\infty}1,$$
which ends the proof of the claim.\\

\noindent Applying Lemma \ref{classical} again, it follows that
 \begin{eqnarray}\label{key-id2}
 \frac{\big|Q_q(1)\big|^4}{2q}+
 \frac1{2q}\sum_{j=1}^{q-1}\big|Q_q(\xi_{q,j})\big|^4+
 \frac1{2q}\sum_{j=0}^{q-1}\big|Q_q(-\xi_{q,j})\big|^4=\|Q_q\|_4^4.
\end{eqnarray}
Hence
\begin{eqnarray}\label{key-id3}
 \frac{\big|Q_q(1)\big|^4}{2q^2}+
 \frac1{2q^2}\sum_{j=1}^{q-1}\big|Q_q(\xi_{q,j})\big|^4+
 \frac1{2q^2}\sum_{j=0}^{q-1}\big|Q_q(-\xi_{q,j})\big|^4=\frac{\|Q_q\|_4^4}{q}.
\end{eqnarray}
Therefore 
\begin{eqnarray}\label{key-id4}
\frac{\big|Q_q(1)\big|^4}{2q^2}+
 \frac1{2q^2}\sum_{j=0}^{q-1}\big|Q_q(-\xi_{q,j})\big|^4 \leq \frac{\|Q_q\|_4^4}{q}.
\end{eqnarray} 
We further notice that for $q$ large enough
\begin{eqnarray}\label{key-id5}
\frac{16|H|^4}{q^3}\sim q.
\end{eqnarray}
At this point we claim that for $q$ large enough we have
$$\frac{1}{q}\sum_{j=0}^{q-1}\big|Q_q(-\xi_{q,j})\big|^4 \sim \frac{1}{3}q+\frac{2}{3q}.$$
Indeed, by the well-known Lagrangian interpolation formula, we have
$$Q(z)=\frac{1}{q}\sum_{j=0}^{q-1}\xi_{q,j}\frac{z^q-1}{z-\xi_{q,j}}Q(\xi_{q,j}).$$
Whence
\begin{eqnarray}\label{interpol1}
Q(-\xi_{q,k})&=&\frac{2}{q}\sum_{j=0}^{q-1}\frac{\xi_{q,j}}{\xi_{q,k}+\xi_{q,j}}Q_q(\xi_{q,j}) \nonumber\\
&=&\frac{2}{q}\frac{1}{1+\xi_{q,k}}Q_q(1)+\frac{2}{q}\sum_{j=1}^{q-1}\frac{\xi_{q,j}}{\xi_{q,k}+\xi_{q,j}}P_q(\xi_{q,j})\\
&=&\frac{2}{q}\frac{Q_q(1)}{1+\xi_{q,k}}+P_q(-\xi_{q,k})-\frac{2}{q}\frac{P_q(1)}{1+\xi_{q,k}} \nonumber
\end{eqnarray}
The second identity follows from \eqref{rootof1}. We thus get from \eqref{Littlewood3} the following identity
\begin{eqnarray}\label{rootf7}
Q(-\xi_{q,k})=\frac{2}{q}\frac{D_q(1)}{1+\xi_{q,k}}+P_q(-\xi_{q,k}).
\end{eqnarray}
Applying the triangle inequalities we obtain
\begin{eqnarray}\label{eqkeyf1}
&&D_q(1)\Big(\frac{1}{q}\sum_{k=0}^{q-1}\frac{16}{q^4}\frac{1}{\big|1+\xi_{q,k}\big|^4}\Big)^\frac{1}{4}
-\Big(\frac{1}{q}\sum_{k=0}^{q-1}|P_q(-\xi_{q,k})|^4\Big)^\frac{1}{4} \leq \nonumber\\
&& \Big(\frac{1}{q} \sum_{k=0}^{q-1}|Q_q(-\xi_{q,k})|^4\Big)^\frac{1}{4},
\end{eqnarray}
and
\begin{eqnarray}\label{eqkeyf2}
&&\Big(\frac{1}{q} \sum_{k=0}^{q-1}|Q_q(-\xi_{q,k})|^4\Big)^\frac{1}{4} \leq \nonumber
\\
&& D_q(1)\Big(\frac{1}{q}\sum_{k=0}^{q-1}\frac{16}{q^4}\frac{1}{\big|1+\xi_{q,k}\big|^4}\Big)^\frac{1}{4}
+\Big(\frac{1}{q}\sum_{k=0}^{q-1}|P_q(-\xi_{q,k})|^4\Big)^\frac{1}{4}.
\end{eqnarray}
Moreover, by taking into account the following estimation from \cite{Hoholdt-Jensen}, \cite[Art.67,68]{Bromwich}\footnote{The identity (HJBr) can be obtained as a consequence of \eqref{four2} by noticing that for every $k=0,\cdots,q-1$,
$$\frac{2}{1+\xi_{q,k}}=\sum_{l=0}^{q-1}(-1)^l \xi_{q,k}^l=R_q(\xi_{q,j}),$$
where $\ds R_q(z)=\sum_{l=0}^{q-1}(-1)^l z^l$. The correlation coefficients are the autocorrelation of the binary sequence
$((-1)^l)$. It can be also obtained by applying \eqref{four2} to the Dirichlet kernel.} 
$$\sum_{k=0}^{q-1}\frac{1}{\big|1+\xi_{q,k}\big|^4}=\frac{1}{16}\Big(\frac{1}{3}q^4+\frac{2}{3}q^2\Big), \eqno(HJBr)$$
we see that 
$$\frac{16}{q^4}\sum_{k=0}^{q-1}\frac{1}{\big|1+\xi_{q,k}\big|^4}=\frac{1}{3}+\frac{2}{3q^2}.$$
That is,
\begin{eqnarray}\label{estimI}
 \frac{1}{q}\sum_{k=0}^{q-1}\frac{16}{q^4}\frac{\big(D_q(1)\big)^4}{\big|1+\xi_{q,k}\big|^4}=\frac{1}{3}q+\frac{2}{3q}.
\end{eqnarray}
Applying again Lemma \ref{classical} combinded with our assumption we infer that the sequence 
$\Big(\frac{1}{q}\sum_{k=0}^{q-1}|P_q(-\xi_{q,k})|^4\Big)$ is bounded.
We thus have the following estimation
\begin{eqnarray}\label{estimII}
 \frac{1}{q}\sum_{k=0}^{q-1}\frac{16}{q^4}\frac{\big(D_q(1)\big)^4}{\big|1+\xi_{q,k}\big|^4}
 \sim \frac{1}{q} \sum_{k=0}^{q-1}|Q_q(-\xi_{q,k})|^4,
\end{eqnarray}
which gives 
\begin{eqnarray}\label{estimIII}
 \frac{1}{q} \sum_{k=0}^{q-1}|Q_q(-\xi_{q,k})|^4 \sim  \frac{1}{3}q+\frac{2}{3q},
\end{eqnarray}
and this finish the proof of the claim.\\

\noindent We deduce that the autocorrelation coefficients $(c_l)$ of $H$ satisfy
$$\sum_{l\neq 0}c_l^2 \sim \frac{2}{3}\frac{q^3}{16}+\frac{1}{3}\frac{q}{16}.$$
In other words
$$\sum_{l\neq 0}|H \cap (H-l)|^2 \sim \frac{2}{3}\frac{q^3}{16}+\frac{1}{3}\frac{q}{16},$$
since 
$$c_l=|H \cap (H-l)|=\1_H * \1_H(-l),$$
where $*$ is the standard convolution operation.\\

\noindent We further observe that the quantity $\sum_{l\neq 0}^{q-1}|H \cap (H-l)|^2$ is the number of the solution of
the equation $a-b=c-d$, $a,b,c, d \in H$.\\
In the same manner we can see that the autocorrelation coefficients $(c'_l)$ of $H^c$ satisfy
$$\sum_{l\neq 0}{c'_l}^2 \sim \frac{2}{3}\frac{q^3}{16}+\frac{1}{3}\frac{q}{16},$$
where $H^c$ is the complement of $H$. It follows that for large enough $q$  
$$\|R_q\|_4 \sim \|Q_q\|_4,$$
where  $\ds R_q(z)=\frac{2}{\sqrt{q}}\sum_{j \in H^c}z^j.$ Hence, by a variant of Balog-Szemer\'edi's theorem, 
we can describe approximately the structure of $H$ \cite{Gowers}. As a consequence, we obtain a contradiction.  
But, we can also proceed directly. Indeed, for each $k=0,\cdots,q-1$,
\begin{eqnarray}\label{key-id6}
|1+\xi_{q,k}|\big|Q_q(-\xi_{q,k})\big| = \Big|\frac{2}{\sqrt{q}}+(1+\xi_{q,k})P_q(-\xi_{q,k})\Big|,
\end{eqnarray} 
by \eqref{rootf7}. Therefore, by the triangle inequality, we can assert
\begin{eqnarray}\label{residu:eq1}
 \quad \quad\frac{1}{2q}\sum_{k=0}^{q-1}|1+\xi_{q,k}|^4\big|Q_q(-\xi_{q,k})\big|^4
\leq \psi(q)+\frac{1}{2q}\sum_{k=0}^{q-1}|1+\xi_{q,k}|^4\big|P_q(-\xi_{q,k})\big|^4,
\end{eqnarray}
where $\psi(q)$ is given by
\begin{eqnarray*}
\psi(q)&=&\frac12 \Big( \Big(\frac{2}{\sqrt{q}}\Big)^4+4 \Big(\frac{2}{\sqrt{q}}\Big)^3 
\Big(\frac{1}{q}\sum_{k=0}^{q-1}|1+\xi_{q,k}|^4\big|P_q(-\xi_{q,k})\big|^4\Big)^{\frac14}\\
&+&
6\Big(\frac{2}{\sqrt{q}}\Big)^2 
\Big(\frac{1}{q}\sum_{k=0}^{q-1}|1+\xi_{q,k}|^4\big|P_q(-\xi_{q,k})\big|^4\Big)^{\frac12}\\
&+& 4\Big(\frac{2}{\sqrt{q}}\Big) 
\Big(\frac{1}{q}\sum_{k=0}^{q-1}|1+\xi_{q,k}|^4\big|P_q(-\xi_{q,k})\big|^4\Big)^{\frac34}. 
\end{eqnarray*}

We further have $\psi(q) \longrightarrow 0$ as $q \longrightarrow +\infty$, by our assumption. Consequently, we have
\begin{eqnarray}\label{residu:f2}
 \frac{1}{2q}\sum_{k=0}^{q-1}|1+\xi_{q,k}|^4\big|Q_q(-\xi_{q,k})\big|^4+
\frac{1}{2q}\sum_{k=0}^{q-1}|1-\xi_{q,k}|^4\big|Q_q(\xi_{q,k})\big|^4 \nonumber\\
\quad \quad \quad \leq\psi(q)+
\frac{1}{2q}\sum_{k=0}^{q-1}|1+\xi_{q,k}|^4\big|P_q(-\xi_{q,k})\big|^4+
\frac{1}{2q}\sum_{k=0}^{q-1}|1-\xi_{q,k}|^4\big|Q_q(\xi_{q,k})\big|^4 
\end{eqnarray}
Combining Lemma \ref{classical} with \eqref{rootof1} we get 

$$\int \big|1-z\big|^4 \big|Q_q\big|^4 dz \leq \psi(q)+\int \big|1-z\big|^4 \big|P_q\big|^4 dz.$$
Whence
\begin{eqnarray}\label{key-majoration}
\int \big|1-z\big|^4 \big|Q_q\big|^4 dz \leq \psi(q)+16\int  \big|P_q\big|^4 dz. 
\end{eqnarray}

Notice here that, without loss of generality, we can suppose $q$ is an even integer. Otherwise, by appealing 
to \eqref{Littlewood3}  we can write 
$$(1-z)D_q(z)+(1-z)P_q(z)=(1-z)Q_q(z),$$
Hence 
\begin{eqnarray}\label{droop1}
\frac{1-z^q}{\sqrt{q}}+(1-z)P_q(z)=(1-z)Q_q(z). 
\end{eqnarray}

Applying the same arguments as before, we see that \eqref{key-majoration} holds with a suitable function $\psi(q)$.\\

We need now to estimate the $L^4$-norm of $(1-z)Q_q(z)$. Write
$$(1-z)Q_q(z)=\frac{2}{\sqrt{q}}\Big(\eta_0+\sum_{j=1}^{q-1}\big(\eta_j-\eta_{j-1}\big) z^j-\eta_{j-1}z^q\Big),$$
and put
$$\tilde{Q_q}(z)=\frac{2}{\sqrt{q}}\Big(\sum_{j=1}^{q-1}\big(\eta_j-\eta_{j-1}\big) z^j\Big).$$
We thus need to estimate only the $L^4$-norm of $\tilde{Q_q}$.\\

Observe that the coefficients of the analytic polynomial 
$\frac{\sqrt{q}}{2}\tilde{Q_q}(z)$ are in $\{0,\pm 1\}$. Furthermore, without loss of generality, we can assume that 
the frequencies of $0$, $1$ and $-1$ exists. For  $a \in \{0,\pm 1\}$,  we denotes by $f_a$ the frequency of $a$.\\

Therefore, as in the proof of Proposition \ref{JJH}, if $f_0<1$ and $f_1 \neq f_{-1}$. Then 
$$\Big\|\tilde{Q_q}(z)\Big\|_4^4 \tend{q}{+\infty}+\infty. \eqno{(GJJH)}$$ Indeed, employing Marcinkiewicz-Zygmund inequalities, we see that
$$\Big\|\tilde{Q_q}(z)\Big\|_4^4 \geq \frac{A_4}{q} \tilde{Q_q}(1)^4.$$
Moreover
$$\tilde{Q_q}(1)=\frac{4}{\sqrt{q}}\#\{j \in \Lambda_0^c :\delta_j=1\}-\frac{2\#\Lambda_0^c }{\sqrt{q}},$$
where $\delta_j=\eta_j-\eta_{j-1}$, $j=1,\cdots, q-1$ and $\Lambda_0^c$ is the complement of the set 
$\big\{j~~:~~\delta_j=0\big\}$.  We thus get for $q$ large enough,
$$\tilde{Q_q}(1) \sim 2 \sqrt{q}\Big(f_1-f_{-1}\Big).$$
Consequently,
$$\Big\|\tilde{Q_q}(z)\Big\|_4^4 \gtrsim  A_4 16 q \Big(f_1-f_{-1}\Big)^4.$$
Letting $q \longrightarrow +\infty$, we obviously obtain $(GJJH)$. We thus need to examine only the case $f_1=f_{-1}$. For that, 
in the same spirit as before,  we complete the proof as follows.\\


Employing Marcinkiewicz-Zygmund inequalities we can assert that 
$$\int \big|1-z\big|^4 \big|Q_q(z)\big|^4 dz \geq \frac{2^4.A_4}{q}Q_q(-1)^4.$$
Therefore if the frequencies of odd integers and even integers in $H$ are not balanced then 
$$\int \big|1-z\big|^4 \big|Q_q(z)\big|^4 dz \tend{q}{+\infty} +\infty.$$
We thus get a contradiction. Otherwise the frequencies are balanced and this also yields a contradiction. Indeed, let $r \geq 2$ be an integer,
and observe once again that we can assume that $r$ divides $q$. Hence again by Marcinkiewicz-Zygmund inequalities, we see that

\begin{eqnarray*}
\Big\|(1-z)Q_q(z)\Big\|_4^4 &\geq& \frac{A_4}{q}\big|1-\xi_{q,\frac{q\ell}{r}}\big|^4\big|Q\big(\xi_{q,\frac{q\ell}{r}}\big) \big|^4\\
&\geq& \frac{A_4}{q}\big|1-\xi_{r,\ell}\big|^4\big|Q\big(\xi_{r,\ell}\big) \big|^4\\
&\geq&
A_42^4 r^4\big|1-\xi_{r,\ell}\big|^4\Big|DF_r\Big(\frac{1}{q}\1_{H \cap [0,q-1]}\Big)(\ell)\Big|^4 q.
\end{eqnarray*}
It follows that if for some $r$ and $\ell \neq 0$, we have 
$$\lim_{q \longrightarrow +\infty} \Big|DF_r\Big(\frac{1}{q}\1_{H \cap [0,q-1]}\Big)(\ell)\Big|>0.$$
Then
$$\int \big|1-z\big|^4 \big|Q_q(z)\big|^4 dz \tend{q}{+\infty}+\infty,$$
which gives a contradiction. Otherwise $H$ mod $r$ is almost equipped with the uniform probability measure of $\big\{0,\cdots,r-1\big\}$,
which also yields a contradiction.\\
Summarizing, we conclude that $(P_q)$ can not be square $L^2$-flat and the proof of the first main result is complete.
\end{proof}

\section{Square flatness implies pairwise independence vs orthogonality.}\label{dynamical}

In this section  we will give an alternative proof of Theorem \ref{mainL1} and we will present the proof of the following theorem 

\begin{Th}\label{Lmain2} If $(\epsilon_j)$ generated a sequence of square $L^2$-flat polynomials then the canonical projections of 
the dynamical system generated by $(\epsilon_j)$ are pairwise independent. 
\end{Th}

We start by presenting a simple proof of Proposition \ref{JJH}.
\begin{proof}[\textbf{Proof of Proposition \ref{JJH}}]By Marcinkiewicz-Zygmund inequalities,
$$\big\|P_q\big\|_4 \geq \frac{A_4}{q} \big|P_q(1)\big|^4 =A_q q \Big|\frac{1}{q}\sum_{j: \epsilon_j=1}-\frac{1}{q}\sum_{j: \epsilon_j=-1}\Big|^4,$$
but, since $(\big\|P_q\big\|_4)_{q \geq 1}$ is boubded, this forces 
$$\Big|\frac{1}{q}\sum_{j: \epsilon_j=1}-\frac{1}{q}\sum_{j: \epsilon_j=-1}\Big| \tend{q}{+\infty}0.$$
Hence the frequencies of $1$ and $-1$ is $\frac12$. 
\end{proof}

It follows that the square flatness implies that the frequencies of $1$ and $-1$ are balanced, that is, the density of $H$ is $\frac{1}{2}$.
We further have that $H$ possesses a certain arithmetical properties. Indeed, we have the following.
\begin{Prop}\label{KeyProp} Let $\ell \geq 1$, then the density of the set $\big(H \cap \big(H+\ell\big)\big)$ is $\frac14.$  
\end{Prop}
For the proof of Proposition \ref{KeyProp}, we  need some tools. We start by proving the following lemma.
\begin{lem}\label{estimationI} For any $\ell \geq 1$, we have 
$$(1-z^{\ell})Q_q(z)= (1-z^{\ell})P_q(z)+\frac{\ds \big(\sqrt{\ell}D_{\ell}(z)\big) (1-z^q)}{\sqrt{q}}.$$ 
\end{lem}
\begin{proof}By (2.4), we have
$$(1-z^{\ell})Q_q(z)=(1-z^{\ell})P_q(z)+(1-z^{\ell})D_q(z).$$
We thus need only to compute the second term of the right-hand of this equation. This can be easily accomplished by a 
straightforward computation as follows
\begin{eqnarray*}
 (1-z^{\ell}) \sqrt{q}D_q(z)&=& (1-z^{\ell}) \Big(\sum_{j=0}^{q-1}z^j\Big)\\
 &=& \sum_{j=0}^{q-1}z^j-\sum_{j=\ell}^{q+\ell-1}z^j\\
 &=& \sum_{j=0}^{\ell-1}z^j-z^q \big(\sum_{j=0}^{\ell-1}z^j\big)\\
 &=& \big(\sum_{j=0}^{\ell-1}z^j\big) \big(1-z^q\big)\\
\end{eqnarray*}
we thereby get 
$$(1-z^{\ell}) \sqrt{q}D_q(z)=\big(\sqrt{\ell}D_{\ell}(z)\big) (1-z^q),$$
 and the lemma follows.
\end{proof}
Lemma \ref{estimationI} and the inequality \eqref{key-majoration} can be improved by proving the following Proposition.
\begin{Prop}Let $\alpha$ be  in $]1,4[$. Then there exist a constant $K_\alpha$ such that for any $q$,
$$\int \big|1-z^{\ell_q}\big|^\alpha \big|Q_q(z)\big|^\alpha dz \leq K_\alpha,  \leqno(BGHJJ)$$
where $\ell_q \leq q^{\frac{\alpha}{2(\alpha-1)}}.$
\end{Prop}
\begin{proof} We start by noticing that we need to prove only that $(BGHJJ)$ holds for $q$ large enough. 
Applying the triangle inequalities combined with Lemma \ref{estimationI}, we  see that 
\begin{eqnarray*}
 \Big\|\big(1-z^{\ell_q}\big) Q(z) \Big\|_{\alpha}  &\leq&  \Big\| (1-z^{\ell_q})P_q(z)\Big\|_{\alpha}+ 
 \Big\|\frac{\ds \big(\sqrt{\ell}D_{\ell_q}(z)\big) (1-z^q)}{\sqrt{q}}\Big\|_{\alpha}\\  
 &\leq& 2 \Big\|P_q(z)\Big\|_{\alpha}+\frac2{\sqrt{q}} \Big\|\sqrt{\ell_q}D_{\ell_q}(z)\Big\|_{\alpha}\\
 & \leq & 2 \Big\|P_q(z)\Big\|_{4}+\frac2{\sqrt{q}} \Big\|\sqrt{\ell_q}D_{\ell_q}(z)\Big\|_{\alpha}.
\end{eqnarray*}
 But, the sequence $\Big(\big\|P_q(z)\big\|_{4}\Big)_{q \geq 1}$ is bounded by our assumption. Henceforth, 
 we need to estimate only the second term in the  right hand side. To this end, 
  by the estimation obtained in \cite{Anderson} (see Remark below), we have 
 $$\Big\|\sqrt{\ell_q}D_{\ell_q}(z)\Big\|_{\alpha}^{\alpha} \sim c_\alpha \ell_q^{\alpha-1} \leq c_\alpha
 q^{\frac{\alpha}{2}}.$$
 We thus get 
 $${\frac2{\sqrt{q}}} \Big\| \sqrt{\ell_q}D_{\ell_q}(z) \Big\|_{\alpha} \leq 2c_{\alpha},$$
 and the proof of the proposition is complete.
\end{proof} 

By applying carefuly Lemma \ref{estimationI} we get
\begin{Prop}\label{KeyProp2}
\begin{math} \ds
 \int \Big|\big|1-z^{\ell_q}\big|^2 \big|Q(z)\big|^2-\big|1-z^{\ell_q}\big|^2\Big| dz \tend{q}{+\infty}0, 
\end{math}
where $\ell_q \leqslant q^{1-\delta}$, $\delta>0.$ 
\end{Prop}
\begin{proof}By the triangle inequality, we have
\begin{eqnarray*}
&&\Big\|\big|1-z^{\ell_q}\big|^2 \big|Q(z)\big|^2-\big|1-z^{\ell_q}\big|^2\Big\|_1 \leqslant  \\
&&\Big\|\big|1-z^{\ell_q}\big|^2 \big(\big|Q(z)\big|^2-\big|P_q(z)\big|^2\big)\Big\|_1 
+\Big\|\big|1-z^{\ell_q}\big|^2 \big(\big|P(z)\big|^2-1\big)\big|^2\Big\|_1 
\end{eqnarray*}
We further have
\begin{eqnarray*}
&&\Big\|\big|1-z^{\ell_q}\big|^2 \big(\big|Q_q(z)\big|^2-\big|P_q(z)\big|^2\big)\Big\|_1\\
&=&\Big\|\big|1-z^{\ell_q}\big|^2 \big(\big|Q_q(z)\big|-\big|P_q(z)\big|\big) \big(\big|Q_q(z)\big|+\big|P_q(z)\big|\big)\Big\|_1 \\
&\leqslant&\Big\|\big|1-z^{\ell_q}\big|^2 \big(\big|Q_q(z)-P_q(z)\big|\big) \big(\big|Q_q(z)\big|+\big|P_q(z)\big|\big)\Big\|_1\\
&\leqslant& \Big\|\big|1-z^{\ell_q}\big|^2 \big|D_q(z)\big| \big(\big|Q_q(z)\big|+\big|P_q(z)\big|\big)\Big\|_1
\end{eqnarray*}
It follows, by Cauchy-Schwartz inequality, that 
$$\Big\|\big|1-z^{\ell_q}\big|^2 \big(\big|Q_q(z)\big|^2-\big|P_q(z)\big|^2\big)\Big\|_1
$$
$$
\leqslant  \Big\|\big(1-z^{\ell_q}\big).D_q(z)\Big \|_2 .\Big\|\big(1-z^{\ell_q}\big)\big(\big|Q_q(z)\big|+\big|P_q(z)\big|\big)\Big\|_2.
$$
Moreover, by our assumption,
$$\Big\|\big(1-z^{\ell_q}\big)\big(\big|Q_q(z)\big|+\big|P_q(z)\big|\big)\Big\|_2
\leq 2 \Big(\big\|Q_q\big\|_2+\big\|P_q\big\|_2\Big)  \lesssim 6,$$
and, we have
$$\Big\|\big(1-z^{\ell_q}\big).D_q(z)\Big \|_2 \leqslant 2\sqrt{\frac{\ell_q}{q}} \tend{q}{+\infty}0,$$
by Lemma 2.1. This achieve the proof of the proposition.
\end{proof}
 
 We can strengthen the previous results by proving
\begin{Prop}\label{KeyProp3}Let $\alpha \in ]0,2[$, $\ell \geq 1$ be a integer and $(\ell_q)$ a sequence of integers. Then 
\vskip 0.01cm
 \begin{enumerate}[(i)]
  \item  \begin{math}
  \ds \int \Big||1-z^\ell|^2|Q_q(z)|^2-|1-z^\ell|^2\Big|^2 dz \tend{q}{+\infty}0.
 \end{math}
 \item \begin{math}
  \ds \int \Big||1-z^{\ell_q}||Q_q(z)|-|1-z^{\ell_q}|\Big|^\alpha dz \tend{q}{+\infty}0,
 \end{math}
 \end{enumerate}
\end{Prop}
\begin{proof}We give the proof of (i) only for the case $\ell = 1$; since the proofs of the other cases are similar. By the triangle inequality, we have
$$\Big\||1-z|^2(|Q_q(z)|^2-1) \Big\|_2$$ 
$$\leqslant \Big\| |1-z|^2(|Q_q(z)|^2-|P_q(z)|^2) \Big\|_2+\Big\| |1-z|^2(|P_q(z)|^2-1) \Big\|_2.$$ 
We further have
$$\Big\| |1-z|^2(|P_q(z)|^2-1) \Big\|_2 \leq 4 \Big\| |P_q(z)|^2-1 \Big\|_2 \tend{q}{+\infty}0,$$
by our assumption. Therefore, the proof of (i) follows once we observe that
\begin{eqnarray*}
 \Big\| |1-z|^2(|Q_q(z)|^2-|P_q(z)|^2) \Big\|_2&=& \Big\| |1-z|\big(|Q_q(z)|-|P_q(z)|\big) |1-z| \big(|Q_q(z)|+|P_q(z)|\big) \Big\|_2\\
 &\leq& \Big\|\frac{|1-z^q|}{\sqrt{q}} |1-z| (|Q_q(z)|+|P_q(z)| \big)\Big\|_2 \\
 &\leq& \frac{2}{\sqrt{q}} \Big(2 \sqrt{\frac{|H|}{q}}+2\Big) \tend{q}{+\infty}0.
\end{eqnarray*}
For (ii),  by the triangle inequalities again,
\begin{eqnarray*}
 \big|1-z^{\ell_q}\big| \Big||Q_q(z)|-1\Big| &\leq&  \big|1-z^{\ell_q}\big| \Big|\big|Q_q(z)\big|-\big|P_q(z)\big|\Big|+
 \big|1-z^{\ell_q}\big| \Big|\big|P_q(z)\big|-1\Big|\\
 &\leq& \big|1-z^{\ell_q}\big| \Big|\big|Q_q(z)-P_q(z)\big|\Big|+\big|1-z^{\ell_q}\big| \Big|\big|P_q(z)\big|-1\Big|\\
 &\leq& 2 \Big|\big|D_q(z)\big|\Big|+\Big|\big|P_q(z)\big|-1\Big|
\end{eqnarray*}
The last inequality is due to the identity \eqref{Littlewood3}. It follows from our assumption that  
$\Big|\big|P_q(z)\big|-1\Big|$ converge almost everywhere to 0 and it is obvious that $(D_q)$ converge to $0$ almost everywhere.
We thus get that for almost all $z$ with respect to $dz$, the sequence $(\big|1-z^{\ell_q}\big| \Big||Q_q(z)|-1\Big|)$ converge to 0. Moreover,
it is $L^\alpha$ uniformly integrable since its $L^2$ norm is bounded by $4$. Indeed, put $u=\frac2{\alpha}, v=\frac2{2-\alpha}$, and apply 
H\"{o}lder inequality to get
\begin{eqnarray*}
 && \int_{\big|1-z^{\ell_q}\big| \Big||Q_q(z)|-1\Big|} \big|1-z^{\ell_q}\big|^\alpha \Big||Q_q(z)|-1\Big|^\alpha dz\\
 &\leqslant& \Big\|\big|1-z^{\ell_q}\big| \Big||Q_q(z)|-1\Big|\Big\|_2^\alpha. \Big(dz\Big\{z~~:~~\big|1-z^{\ell_q}\big| \Big||Q_q(z)|-1\Big| \geq M\Big\}
 \Big)^{\frac1{v}}\\
 &\leqslant& 4^\alpha \Big(dz\Big\{z~~:~~\big|1-z^{\ell_q}\big| \Big||Q_q(z)|-1\Big| \geq M\Big\}
 \Big)^{\frac1{v}}. 
\end{eqnarray*}
We thus get, by Markov inequality,  

$$dz\Big\{z~~:~~\big|1-z^{\ell_q}\big| \Big||Q_q(z)|-1\Big| \geq M\Big\}\leqslant 4\frac1{M}\tend{M}{+\infty}0.$$
This achieve the proof of the claim.  We thus get by the classical Vitali convergence Theorem, that 
$$\Big\|\big|1-z^{\ell_q}\big| \Big||Q_q(z)|-1\Big\|_\alpha \tend{q}{+\infty}0,$$
which finish the proof of the proposition.
\end{proof}

Let us now present the proof of Proposition \ref{KeyProp}.

\begin{proof}[\textbf{Proof of Proposition \ref{KeyProp}}] We begin by noticing that $H \Delta (H+\ell)=\Lambda_0^c,$
where $\Lambda_0=\Big\{j~~:~~\eta_j = \eta_{j-\ell}\Big\}$. We further have
$$\int |1-z^\ell|^2|Q_q(z)|^2 dz =\frac{4|\Lambda_0^c \cap [0,q-1]|}{q} \tend{q}{+\infty} 4 d(\Lambda_0^c),$$
where $d(A)$ is the density of the set $A$. Hence
$d(\Lambda_0^c)=\frac12$, by Proposition \ref{KeyProp3}. Therefore,
\begin{eqnarray}\label{keycomputation}
d(H)+d(H+\ell)-2d(H \cap H+\ell)=1-2d(H \cap H+\ell)=\frac{1}{2}. 
\end{eqnarray}

We thus get $d(H \cap H+\ell)=\frac14.$ This finish the proof of the proposition.
 \end{proof}

\paragraph{\textbf{From Flatness to pairwise independence.}}Of course Proposition \ref{KeyProp3} yields that the sequence 
$(\eta_j)$ generated a pairwise independent process. But we can prove directly that $(\epsilon_j)$ generated 
 a pairwise independent process by proving Theorem \ref{Lmain2}. For that we start by proving the following lemma.

\begin{lem}\label{KeyLem1}Let $\ell \geq 1$ be a integer . Then 
$$
\Big\|\Big|\frac{1}{\sqrt{q}}\sum_{j=\ell}^{q+\ell-1}\epsilon_{j-\ell}z^j\Big|^2-1\Big\|_2  \tend{q}{+\infty}0.
$$
\end{lem}
\begin{proof}The proof of the lemma follows once we observe that 
$$\Big|\frac{1}{\sqrt{q}}\sum_{j=\ell}^{q+\ell-1}\epsilon_{j-\ell}z^j\Big|=
\Big|z^{-\ell}\frac{1}{\sqrt{q}}\sum_{j=0}^{q-1}\epsilon_{j}z^j\Big|.$$
\end{proof}
Consequently, we have the following proposition.

\begin{Prop}\label{KeyProp2}Let $\ell \geq 1$ be a integer. Then 
$$
\Big\|\Big|\frac{1}{\sqrt{q}}\sum_{j=\ell}^{q+\ell-1}\big(\epsilon_j-\epsilon_{j-\ell}\big)z^j\Big|^2-\big|1-z^\ell\big|^2\Big\|_1  \tend{q}{+\infty}0.
$$
\end{Prop}
\begin{proof}
 We start by noticing that we have 
 \begin{eqnarray*}
  \frac{1}{\sqrt{q}}\sum_{j=\ell}^{q+\ell-1}\big(\epsilon_j-\epsilon_{j-\ell}\big)z^j &=&
  \frac{1}{\sqrt{q}}\sum_{j=\ell}^{q+\ell-1}\epsilon_j z^j-z^{\ell}\Big(\frac{1}{\sqrt{q}}\sum_{j=0}^{q-1}\epsilon_jz^j\Big)\\  
  &=&\Big(\frac{1}{\sqrt{q}}\sum_{j=\ell}^{q-1}\epsilon_jz^j\Big)\big(1-z^\ell)+\frac{1}{\sqrt{q}}\sum_{j=q}^{q+\ell-1}\epsilon_jz^j-
 \frac{z^\ell}{\sqrt{q}}\sum_{j=0}^{\ell-1}\epsilon_jz^j
 \end{eqnarray*}
Whence
\begin{eqnarray*}
 &&  \Big|\frac{1}{\sqrt{q}}\sum_{j=\ell}^{q+\ell-1}\big(\epsilon_j-\epsilon_{j-\ell}\big)z^j \Big|^2\\
 &=&\Big|\frac{1}{\sqrt{q}}\sum_{j=\ell}^{q-1}\epsilon_jz^j\Big|^2\big|1-z^\ell\big|^2+
 \Big|\frac{1}{\sqrt{q}}\sum_{j=q}^{q+\ell-1}\epsilon_jz^j-
 \frac{z^\ell}{\sqrt{q}}\sum_{j=0}^{\ell-1}\epsilon_jz^j\Big|^2-2I_q(z),
\end{eqnarray*}
where 
$$I_q(z)= \re{\Big\{
\Big(\frac{1}{\sqrt{q}}\sum_{j=\ell}^{q-1}\epsilon_jz^j\Big)\big(1-z^\ell\big) \Big(\frac{1}{\sqrt{q}}\sum_{j=q}^{q+\ell-1}\epsilon_jz^{-j}-
 \frac{z^\ell}{\sqrt{q}}\sum_{j=0}^{\ell-1}\epsilon_jz^{-j}\Big)\Big\}}.$$
Now, applying the triangle inequality, we obtain
\begin{eqnarray*}
&&\Big\|\Big|\frac{1}{\sqrt{q}}\sum_{j=\ell}^{q+\ell-1}\big(\epsilon_j-\epsilon_{j-\ell}\big)z^j\Big|^2-\big|1-z^\ell\big|^2\Big\|_1\\
&\leq& \Big\|\Big(\Big|\frac{1}{\sqrt{q}}\sum_{j=\ell}^{q-1}\epsilon_jz^j\Big|^2-1\Big)\big|1-z^{\ell}\big|^2\Big\|_1
+\frac{2\ell}{\sqrt{q}}+\Big(\big\|P_q\big\|_1+\frac{\ell}{\sqrt{q}}\Big)\big|\frac{4\ell}{\sqrt{q}}.
\end{eqnarray*}
Letting $q \longrightarrow +\infty$, we get  the desired convergence. The proof of the proposition is complete. 
\end{proof}

We are now able to see that $(\epsilon_j)$ generated a pairwise independent process. 

\begin{proof}[\textbf{Proof of the claim}] It is suffices to show that for any $\ell \geq 1$, we have
\begin{eqnarray*}
 \frac{1}{q}\sum_{j=0}^{q-1}\epsilon_j \epsilon_{j-\ell} \tend{q}{+\infty}0.
\end{eqnarray*}
But, by Proposition \ref{KeyProp2}, we have
$$
\Big\|\frac{1}{\sqrt{q}}\sum_{j=0}^{q-1}\big(\epsilon_j-\epsilon_{j-\ell}\big)z^j\Big\|_2  \tend{q}{+\infty}
\Big\|1-z^\ell\Big\|_2.
$$
Therefore
$$
\frac{1}{q}\sum_{j=0}^{q-1}\big(\epsilon_j-\epsilon_{j-\ell}\big)^2 \tend{q}{+\infty} 2.$$
We further have 
$$
\frac{1}{q}\sum_{j=0}^{q-1}\big(\epsilon_j-\epsilon_{j-\ell}\big)^2
=\frac{1}{q}\sum_{j=0}^{q-1}\big(\epsilon_j^2+\epsilon_{j-\ell}^2\big)-2
\frac{1}{q}\sum_{j=0}^{q-1}\epsilon_j\epsilon_{j-\ell}.$$
We thus get
$$ \frac{1}{q}\sum_{j=0}^{q-1}\epsilon_j\epsilon_{j-\ell}  \tend{q}{+\infty} 0,$$
and this finish the proof of the proposition.
 \end{proof}
Notice that we have proved that the spectral measure of the sequence $(\epsilon_j)$ is a Lebesgue measure.\\

We remind that the notion  of  spectral measure for sequences is introduced by Wiener in
his 1933 book \cite{Wiener}. Therein, he considers the space $S$ of
complex bounded sequences $x =(x_{n})_{n \in \N}$ such that
\begin{equation}\label{Sspace}
\lim_{N \longrightarrow +\infty}
\frac{1}{N}\sum_{n=0}^{N-1}x_{n+k}\overline{x}_{n}=\gamma(k)
\end{equation}
exists for each integer $k \in \N$. The sequence $\gamma(k)$ can be extended to negative
integers by setting
\[
\gamma(-k)=\overline{\gamma(k)}.
\]
It is well known that $\gamma$ is positive definite on $\Z$ and therefore (by
Herglotz-Bochner theorem) there exists a unique positive finite measure $\sigma_g$ on
the circle $\T$ such that the Fourier
coefficients of $\sigma_x$ are given by the sequence $\gamma$.
Formally, we have
\[
\widehat{\sigma_x}(k)\stackrel{\rm{def}}{=}
\int_{\T} e^{-ik t} d\sigma_{x}(t) = \gamma(k).
\]
The measure $\sigma_x$ is called the {\em spectral measure of the sequence $x$}.\\

Summarizing, we have proved the following
\begin{Th}\label{KeyTheorem}
 If $(P_n)$ is square flat. Then, the spectral measure of $\big(\eta_j-\frac{1}{2}\big)$ is a multiple of Lebesgue measure. 
\end{Th}

Now, let us establish the following crucial Theorem.
\begin{Th}\label{KeyTh2}If $(P_q)$ is square flat then the associated Newman-Bourgain polynomials $(Q_q)$ satisfy 
$\ds \Big(\frac{1}{q}\sum_{j=1}^{q-1}\big|Q\big(\xi_{q,j}\big)\big|^4\Big)_{q \geq 1}$
does not converge to $1$.
\end{Th}
\begin{proof}We start by noticing that for any $l \in \Z$ we have 
\[
\frac{1}{q}\sum_{j=1}^{q-1}\widehat{\delta_{\xi_{q,j}}}(l)=
\frac{1}{q}\sum_{j=1}^{q-1} \xi_{q,j}^l=\begin{dcases*}
\frac{-1}{q} & if $l \not \equiv 0$ mod $q$\\
\frac{q-1}{q} & if not.
\end{dcases*}
\]
Therefore, by applying \eqref{four1}, we obtain
\begin{eqnarray}\label{curcial1}
 \frac{1}{q}\sum_{j=1}^{q-1} \big|Q_q(\xi_{q,j})\big|^4&=&\Big(c_0^2+2\sum_{k=1}^{q-1}c_k c_{q-k}+2\sum_{k=1}^{q-1}c_k^2\Big).\frac{q-1}{q}\\
&&-\frac{2}{q}c_0 \sum_{l=1}^{q-1}c_l-\frac{2}{q}\sum_{l+k \neq q}c_l c_k-\frac{2}{q}\sum_{l \neq k }c_l c_k, \nonumber
\end{eqnarray}
where $(c_k)$ are autocorrelation of the sequence $(\frac{2}{\sqrt{q}} .\eta_j)_{j=0}^{q-1}$. We further have 
\begin{eqnarray*}
\sum_{l=1}^{q-1}c_l&=&\frac{2}{q}\sum_{i \neq j} \eta_i \eta_j\\
&=&\frac{2}{q} \Big( \Big(\sum_{i=0}^{q-1} \eta_i\Big)^2-\sum_{i=0}^{q-1} \eta_i\Big)
\end{eqnarray*}
Hence, according to our assumption,
\begin{eqnarray}\label{curcial2}
 \sum_{l=1}^{q-1}c_l \sim \frac{q-2}{2},~~~c_0 \sim 2,
\end{eqnarray}
and 
\begin{eqnarray}\label{curcial3}
c_0^2+2\sum_{k=1}^{q-1}c_k^2 \sim \frac{2}{3}q+\frac{1}{3q}.
\end{eqnarray}
 We thus need to estimate the following quantity 
$$I_1=-\frac{2}{q}c_0 \sum_{l=1}^{q-1}c_l-\frac{2}{q}\sum_{l+k \neq q}c_l c_k-\frac{2}{q}\sum_{l \neq k }c_l c_k.$$
But 
$$ \sum_{l+k \neq q}c_l c_k= \Big(\sum_{l=1}^{q-1}c_l\Big)^2-\sum_{l+k=q}c_lc_k,$$
and 
$$\sum_{l \neq k}c_lc_k=\Big(\sum_{l=1}^{q-1}c_l\Big)^2-\sum_{l=1}^{q-1}c_l^2.$$
Whence
$$I_1=-\frac{2}{q}c_0 \sum_{l=1}^{q-1}c_l-\frac{4}{q}\Big(\sum_{l=1}^{q-1}c_l\Big)^2+\frac{2}{q}\sum_{k=1}^{q-1}c_k c_{q-k}
+\frac{2}{q}\sum_{l=1}^{q-1}c_l^2.$$
Consequently, we need to estimate only $\ds 2\sum_{k=1}^{q-1}c_k c_{q-k}+2\sum_{l=1}^{q-1}c_l^2.$ To this end, we notice that \eqref{four1} combined with our assumption and \eqref{rootof1} gives
\begin{eqnarray}\label{Finalestimate}
\frac{1}{q}\sum_{j=0}^{q-1} \big|Q_q(\xi_{q,j})\big|^4&=&\Big(c_0^2+2\sum_{k=1}^{q-1}c_k c_{q-k}+2\sum_{k=1}^{q-1}c_k^2\Big) \nonumber\\
&=&\frac{Q(1)^4}{q}+\frac{1}{q}\sum_{j=1}^{q-1} \big|P_q(\xi_{q,j})\big|^4 \nonumber \\ 
&\sim& q+\psi(q), 
\end{eqnarray}
where $\psi(q)$ is a bounded sequence.
Combining this with \eqref{curcial2} and \eqref{Finalestimate}, it follows that we have the following estimation 
\begin{eqnarray}\label{I1estim}
I_1 &\sim& -\frac{2}{q}.2.\frac{q-2}{2}-\frac{4}{q}.\Big(\frac{q-2}{2}\Big)^2+\frac{q+\psi(q)}{q}-\frac{c_0^2}{q}\\
&\sim& 3-q+ \frac{\psi(q)}{q}-\frac{4}{q}. \nonumber
\end{eqnarray}
Summarizing, we obtain the following estimation
$$\frac{1}{q}\sum_{j=1}^{q-1}\big|Q\big(\xi_{q,j}\big)\big|^4 \sim 2+\frac{q-1}{q}\psi(q)+\frac{\psi(q)}{q}-\frac{4}{q}.$$
Letting $q \longrightarrow +\infty$ it follows that 
$$\frac{1}{q}\sum_{j=1}^{q-1}\big|Q\big(\xi_{q,j}\big)\big|^4 \longrightarrow 3.$$
This contradicts our assumption in view of \eqref{rootof1}, and the proof of the theorem is finished. 
\end{proof}
In the next subsection, we will present a dynamical proof of our main result Theorem \ref{mainL1}.\\
\subsection{Dynamical proof of the main Theorem \ref{mainL1}}
Let us consider the subshift $(X_H,S)$ generated by $\eta=(\eta_j)$ where $S$ is the shift map on $\{0,1\}^{\Z}$ and 
$X_H$ is the closure of the orbit of $\eta$ under the shift transformation $S$. Let $\P$ be a weak limit measure in the weak closure of
the sequence of the empiric measures $\Big(\frac{1}{N}\sum_{j=1}^{N}\delta_{S^j(\eta)}\Big)$, where 
$\delta_x$ is a Dirac measure on $x$.\\

According to Theorem \ref{KeyTheorem}, we claim that the family of 
coordinates projections $\big(\pi_k\big)_{k \in \Z}$ are pairwise independent under $\P$. Indeed, for any $k \geq 1$, we have 
$$\frac{1}{N}\sum_{j=0}^{N-1}\pi_0\big(S^j(\eta)\big) \pi_{k}\big(S^j(\eta)\big)
=\frac{1}{N}\sum_{j=0}^{N-1}\eta_j \eta_{j+k}.$$
Letting $N\longrightarrow +\infty$, we obtain
$$
\frac{1}{N}\sum_{j=0}^{N-1}\pi_0\big(S^j(\eta)\big) \pi_{k}\big(S^j(\eta)\big)
\tend{N}{+\infty} \int \pi_0(x) \pi_k(x) d\P(x).$$
We further obtain, by Theorem \ref{KeyTheorem}, 
$$
\frac{1}{N}\sum_{j=0}^{N-1}\pi_0\big(S^j(\eta)\big) \pi_{k}\big(S^j(\eta)\big)
\tend{N}{+\infty} \int \pi_0(x) d\P(x) \int \pi_k(x) d\P(x).$$
This finish the proof of the claim. Now, the proof of our main result will follows from the following theorem.
\begin{Th}\label{pairwise-ind}Let $(X_n)$ be a stationary sequence of pairwise independent
random variables taking values in $\{0,1\}$. Then the sequence of random analytic polynomials 
$Q_q(z) =\frac{2}{\sqrt{q}}\sum_{j=0}^{q-1}X_j z^j$ satisfy, for $q$ large enough,
$$\frac{1}{q}\sum_{j=1}^{q-1}\int \Big|Q_q(\xi_{q,j})\Big|^ 4 d\P >1.$$ 
\end{Th}
\begin{proof}Analysis similar to that in the proof of Theorem \ref{KeyTh2} shows
that we have
\begin{eqnarray}\label{eqfirst}
\frac{1}{q}\sum_{j=1}^{q-1}\int \big|Q_q(\xi_{q,j})\big|^4 d\P&&=\frac{q-1}{q}
\Big(\int c_0^2d\P+2\sum_{k=0}^{q-1}\int c_k c_{q-k} d\P +2 \sum_{k=0}^{q-1}\int c_k^2 d\P \Big) \nonumber\\
&&-\frac{2}{q}\int c_0 \sum_{l=1}^{q-1}c_l d\P-\frac{4}{q} \int \Big(\sum_{l=1}^{q-1}c_l\Big)^2 d\P \\
&&+\frac{1}{q} \Big(2\sum_{k=0}^{q-1}\int c_k c_{q-k} d\P +2 \sum_{k=0}^{q-1}\int c_k^2 d\P \Big) \nonumber,  
\end{eqnarray}
where $(c_k)$ are autocorrelation of the sequence $(\frac{2}{\sqrt{q}} .X_j)_{j=0}^{q-1}$.
We thus need to estimate
$$\int c_0^2 d\P+2\sum_{k=1}^{q-1}\int c_k c_{q-k} d\P+2 \sum_{k=0}^{q-1} \int c_k^2 d\P,$$
and
$$\frac{2}{q}\int c_0 \sum_{l=1}^{q-1}c_l d\P+\frac{4}{q} \int \Big(\sum_{l=1}^{q-1}c_l\Big)^2 d\P.$$
But, for any $k=1,\cdots,q-1,$ we have
$$ \int c_{q-k} d\P=\frac{k}{q},$$
and
$$\int c_0^2 d\P= \frac{8}{q}+\frac{4(q-1)}{q},$$
We further have 
$$2\sum_{k=1}^{q-1}\int c_k c_{q-k} d\P+2 \sum_{k=0}^{q-1} \int c_k^2d\P
=\sum_{k=1}^{q-1}\int \big(c_k+c_{q-k}\big)^2 d\P,$$
and 
$$\frac{2}{q} \sum_{l=1}^{q-1}c_l =\Big(\frac{2}{q}\sum_{j=0}^{q-1}X_j\Big)^2-\frac4{q^2}\sum_{j=0}^{q-1}X_j ,$$
Therefore, it is suffices to estimate $\ds  
\sum_{k=1}^{q-1} \int \big(c_k+c_{q-k}\big)^2 d\P$, and $\ds \frac{1}{q}\sum_{j=1}^{q-1}X_j.$\\
For that, we apply Cauchy-Schwarz inequality to estimate the first quantity as follows 
\begin{eqnarray*}
\sum_{k=1}^{q-1} \int \big(c_k+c_{q-k}\big)^2 d\P \geq \sum_{k=1}^{q-1} \Big(\int \big(c_k+c_{q-k}\big) d\P\Big)^2=q-1.
\end{eqnarray*}
To estimate the second quantity, we notice that according to Theorem 5.2 from \cite[pp.158]{Doob}, 
the strong law of large numbers holds for the sequence 
of random variables $(X_j-\frac12)$. Therefore 
$$\frac{1}{q}\sum_{j=0}^{q-1}X_j \tend{q}{+\infty}\frac{1}{2},$$
almost surely (a.s.). This combined with the Lebesgue's dominated convergence theorem gives 
$$\int \Big(\frac{2}{q}c_0\sum_{l=1}^{q-1}c_l\Big) d\P \tend{q}{+\infty} 2,$$
since 
$$c_0=\frac{4}{q}\sum_{j=0}^{q-1}X_j \xrightarrow[q \rightarrow +\infty]{\textrm{a.s.}} 2,$$
and
$$\frac{2}{q}\sum_{l=1}^{q-1}c_l \xrightarrow[q \rightarrow +\infty]{\textrm{a.s.}}   1.$$
Applying once again Cauchy-Schwarz inequality we obtain
$$\Big(\sum_{k=1}^{q-1}\big(c_k+c_{q-k}\big)\Big)^2 \leq q.\Big(\sum_{k=1}^{q-1} \big(c_k+c_{q-k})^2\Big).$$
Whence 
\begin{eqnarray}\label{eqII}
\int \Big(\sum_{k=1}^{q-1} \big(c_k+c_{q-k})^2\Big) d\P \geq \frac{4}{q} \int \Big(\sum_{k=1}^{q-1}c_k\Big)^2 d\P. 
\end{eqnarray}
Combining \eqref{eqfirst} with \eqref{eqII} we can rewrite \eqref{eqfirst} as
\begin{eqnarray}\label{eqSIII}
\frac{1}{q}\sum_{j=1}^{q-1}\int \big|Q_q(\xi_{q,j})\big|^4 d\P&& \geq \frac{q-1}{q}
\Big(\int c_0^2d\P+ \frac{4}{q} \int \Big(\sum_{k=1}^{q-1}c_k\Big)^2 d\P \Big)\\
&&-\frac{2}{q}\int c_0 \sum_{l=1}^{q-1}c_l d\P-\frac{4}{q} \int \Big(\sum_{l=1}^{q-1}c_l\Big)^2 d\P \nonumber\\
&&+\frac{1}{q} \Big(2\sum_{k=0}^{q-1}\int c_k c_{q-k} d\P +2 \sum_{k=0}^{q-1}\int c_k^2 d\P \Big) \nonumber,  
\end{eqnarray} 
Thus, a straightforward calculation yields
\begin{eqnarray}\label{eqbon}
\frac{1}{q}\sum_{j=1}^{q-1}\int \big|Q_q(\xi_{q,j})\big|^4 d\P&&\geq \frac{q-1}{q}\int c_0^2 d\P-\frac{4}{q^2} 
\int \Big(\sum_{k=1}^{q-1}c_k\Big)^2 d\P-\frac{2}{q}\int c_0 \sum_{l=1}^{q-1}c_l d\P \nonumber\\
&&+\frac{1}{q} \Big(2\sum_{k=0}^{q-1}\int c_k c_{q-k} d\P +2 \sum_{k=0}^{q-1}\int c_k^2 d\P \Big) \nonumber,  
\end{eqnarray} 
It remains to estimate $\ds \frac{4}{q^2} \int \Big(\sum_{l=1}^{q-1}c_l\Big)^2 d\P$. This estimate can be obtained in the same manner as 
before. Indeed, applying once again the strong large numbers and the Lebesgue's dominated convergence theorem, we get   
 $$\int \frac{4}{q^2} \Big(\sum_{l=1}^{q-1}c_l\Big)^2 d\P \tend{q}{+\infty} 1.$$ 
Summarizing, it follows that 
\begin{eqnarray*}
\frac{1}{q}\sum_{j=1}^{q-1}\int \big|Q_q(\xi_{q,j})\big|^4 d\P 
&&\geq \frac{q-1}{q}\Big(\frac{16}{q}+\frac{4(q-1)}{q}\Big)-
\frac{4}{q^2} \int \Big(\sum_{k=1}^{q-1}c_k\Big)^2 d\P\\
&&-\frac{2}{q}\int c_0 \sum_{l=1}^{q-1}c_l d\P+\frac{q-1}{q}\\
\end{eqnarray*}

Letting $q \longrightarrow +\infty$, we see that for sufficiently large $q$, 
$$ \frac{1}{q}\sum_{j=1}^{q-1}\int \big|Q_q(\xi_{q,j})\big|^4 d\P  \gtrsim 2,$$
and this finishes the proof of the theorem.\\ 

As a corollary, for the  mutually independent random variables, we obtain that for a large enough $q$, we have
\begin{eqnarray}\label{diverge}
\frac{1}{q}\sum_{j=1}^{q-1}\int \Big|Q_q(\xi_{q,j})\Big|^ 4 d\P > 1. 
\end{eqnarray}

Of course, in this case one can compute exactly $\ds \int c_{k}^2 d\P$, for each $k=1,\cdots,q-1,$ and all the terms in 
$\ds \frac{1}{q}\sum_{j=1}^{q-1}\int \big|Q_q(\xi_{q,j})\big|^4 d\P $. Nevertheless, one has to be 
careful since the terms in $c_{q-k}$, that is, $Y_j=X_j X_{j+q-k}$, $j=0,\cdots,k-1$ are not in general mutually independent. Indeed, 
if we take  $k=q-1$, we get $Y_j=X_j X_{j+1}$ which obviously are not independent. However, it is a simple matter to compute explicitly 
all the terms in $\ds \frac{1}{q}\sum_{j=1}^{q-1}\int \big|Q_q(\xi_{q,j})\big|^4 d\P$ and to see that \eqref{diverge} holds.    
This allows us to conclude that the sequence of random polynomials $(P_q)$ can not be square flat.\\

This allows us to  obtain a new proof 
of the well-known result of Newman-Byrnes \cite{Byrnes-Newman} which say 
that the random polynomials trigonometric with Rademacher coefficients are not square flat. \end{proof}

The proof of Theorem \ref{mainL2}  and \ref{mainL3} is straightforward from Theorem \ref{mainL1}, 
since the spectrum of Morse cocycle satisfy the purity law which say that the spectrum 
is either equivalent to the Lebesgue measure on $S^1$ or singular to it. For the proof of this last fact we refer to \cite{Helson}
or \cite[p.73-80]{Queffelec}.

\section{an application to Number Theory: Liouville function, Chowla conjecture and Riemann hypothesis}\label{ChowlaRH}

In this section we choose the sequence $(\epsilon_j)$ to be the Liouville function. The Liouville function 
$\lamob$ is given by 
\begin{equation*}\label{Liouville}
\lamob(n)= \begin{cases}
 1 {\rm {~if~}} n=1; \\
(-1)^r  {\rm {~if~}} n
{\rm {~is~the~product~of~}} r {\rm {~not~necessarily~distinct~prime~numbers}}; 
\end{cases}
\end{equation*}
The Liouville function is related to another famous functions in number theory called the M\"{o}bius function. Indeed, 
the M\"{o}bius function is defined for the positive integers $n$ by
\begin{equation*}\label{Mobius}
\mob(n)= \begin{cases}
 \lamob(n) {\rm {~if~}} n \rm{~is~not~divisible~by~the~square~of~any~ prime}; \\
0  {\rm {~if~not}}
\end{cases}
\end{equation*}
Those two functions are of great importance in number theory because of its  connection with the
Riemann $\zeta$-function via the formulae
\[
\sum_{n=1}^{+\infty}\frac{\mob(n)}{n^s}=\frac{1}{\zeta(s)}, \qquad
\sum_{n=1}^{+\infty}\frac{\lamob(n)}{n^s}=\frac{\zeta(2s)}{\zeta(s)}
~~{\rm {with}}~~~ \re{(s)}>1,
\]
and 
\[
\sum_{n=1}^{+\infty}\frac{|\mob(n)|}{n^s}=\frac{\zeta(s)}{\zeta(2s)}
~~{\rm {with}}~~~ \re{(s)}>1.
\]
Let us further notice that the Dirichlet inverse of the Liouville function is the absolute value of the M\"{o}bius function.\\

For the reader's convenience, we briefly remind some useful well-known results on the Riemann $\zeta$-function. The  Riemann $\zeta$-function is defined, 
for $s \in \C$, $\re(s)>1$ by 
$$\zeta(z)=\sum_{n=1}^{+\infty}\frac{1}{n^s},$$
or by the Euler formula
$$\zeta(s)=\prod_{\overset{p}{\textrm{~prime}}}\Big(1-\frac{1}{p^s}\Big)^{-1}.$$
It is easy to check that $\zeta$ is analytic for $\re(s)>1$. Moreover, it is well-known that $\zeta$ is regular for all values of $s$ except $s=1$, where there is a
simple pole with residue 1. Thanks to the functional equation 
$$\zeta(s)=2^{s}\pi^{s-1}\sin\Big(\frac{\pi s}{2}\Big) \Gamma(1-s)\zeta(1-s),$$
where $\Gamma$ is the gamma function given by 
$$\Gamma(z)=\int_{0}^{+\infty}x^{z-1}e^{-x}dx,~~~~~~~~~\re(z)>0.$$ We notice that the gamma function never vanishes and it is analytic 
everywhere except at $z=0, -1, -2, ...,$ with the residue at $z=-k$ is equal to
$\frac{(-1)^k}{k!}.$  We further have the following  formula (useful in the proof of the functional equation) 
$$\Gamma(s) \sin\Big(\frac{\pi s}{2}\Big)=\int_{0}^{+\infty} y^{s-1} \sin\big(y\big) dy.$$
For the proof of it we refer to \cite[p.88]{Redmancher}. Changing $s$ to $1-s$, we obtain
$$\zeta(1-s)=2^{1-s}\pi^{-s}\cos\Big(\frac{\pi s}{2}\Big) \Gamma(s)\zeta(s).$$
Putting
$$\xi(s)=\frac{s(s-1)}{2}\pi^{\frac{-s}{2}}\Gamma\big(\frac{s}{2}\big)\zeta(s),$$
and 
$$E(s)=\xi\Big(\frac{1}{2}+is\Big).$$
It follows that
$$\xi(s)=\xi(1-s),\quad \quad \textrm{and} \quad \quad E(z)=E(-z).$$  
We further remind that we have 
$$\zeta(s)\Gamma(s)=\int_{0}^{+\infty}\frac{x^{s-1}}{e^x-1} dx,\quad \quad \re(s)>1.$$
Therefore, it is easy to check that $\zeta$ has no zeros for $\re(s)>1$. It follows also from the functional equation that $\zeta$ has no zeros for $\re(s)<0$ except for simple zeros at 
$s=-2,-4, \cdots$. Indeed, $\zeta(1-s)$ has no zeros for $\re(s)<0$, $\sin\big(\frac{s\pi}{2}\big)$ has simple zeros at $s=-2,-4,\cdots$.  
It is also a simple matter to see that 
$\xi(s)$ has no zeros for $\re(s)>1$ or $\re(s)<0$. Hence its  zeros which are also the zeros of $\zeta$ lie in the 
strip $0 \leq \re(s) \leq 1$. Notice that for $\re(s)>1$,  it is easily seen that 

\begin{eqnarray*}
\sum_{n=1}^{+\infty}\frac{(-1)^{n-1}}{n^s}&=&\sum_{n=1}^{+\infty}\frac{1}{n^s}-2\sum_{n=1}^{+\infty}\frac{1}{2^s n^s} \\
&=&\big(1-2^{1-s}\big)\zeta(s),
\end{eqnarray*}

This formula allows us to continue $\zeta$ analytically to half-plan $\re(s)>0$ with simple pole at $s=1$. 
We further have $\zeta(s)\neq 0$ for all $s>0$ since 
$ \sum_{n=1}^{+\infty}\frac{(-1)^{n-1}}{n^s}>0.$\\

We thus conclude that all zeros of $\zeta$ are complex. 
The functional equation allows us also to see that if $z$ is a zero then $1-z$ and 
$1-\overline{z}$ are also a zeros. Whence, the zeros of $\zeta$ lie on the vertical line $\re(s)=\frac{1}{2}$ or occur in pairs symmetrical about this line.\\

{\it{The Riemann hypothesis assert that all nontrivial zeros of $\zeta$ lie on the critical line  $\re(s)=\frac{1}{2}$.}}\\   

We further have that the estimate
\begin{eqnarray}\label{RH1}
\left|\ds \sum_{n=1}^{x}\mob(n)\right|=O\left(x^{\frac12+\varepsilon}\right)\qquad
{\rm as} \quad  x \longrightarrow +\infty,\quad \forall \varepsilon >0
\end{eqnarray}

is equivalent to the Riemann Hypothesis (\cite[pp.370, Theorem 14.25(B)]{Titchmarsh}). Following Chowla \cite{Chowla}, 
this result is due to Littlewood (see also \cite[section 2.12,p. 261]{Edwards}).\\

\noindent{}Combining this result with Batman-Chowla trick \cite{Batman-Chowla}, it can be shown that  
\begin{eqnarray}\label{RH2}
\left|\ds \sum_{n=1}^{x}\lamob(n)\right|=O\left(x^{\frac12+\varepsilon}\right)\qquad
{\rm as} \quad  x \longrightarrow +\infty,\quad \forall \varepsilon >0
\end{eqnarray}
is equivalent to the Riemann Hypothesis.\medskip

There is many problems and conjectures about the Liouville and M\"{o}bius functions in number theory, combinatorics and dynamical systems. 
But, the more famous one  are the two following conjectures of Chowla.

\begin{Conj}(\cite[problem 57, pp.]{Chowla})\label{ChowlaI}
Let $f(x)$ be an arbitrary polynomial with integer coefficients,
which is not, however, of the form $cg^2(x)$, where $c$ is an integer and
$g(x)$ is a polynomial with integer coefficients. Then

$$ \sum_{n=1}^{x}\lamob(f(n))=O(x).$$
\end{Conj}

\begin{Conj}(\cite[problem 56, pp.96]{Chowla})\label{ChowlaII}
Let $\epsilon_1,\epsilon_2,\cdots, \epsilon_g$ be arbitrary numbers each equal to $+1$ or
$-1$, where $g$ is a fixed (but arbitrary) number. Then the equations (in $n$)
$$ \lamob(n + m) = \epsilon_m,~~~~~~~~~~ (1 \leq m \leq g)$$
have infinitely many solutions.
\end{Conj}
This later conjecture holds if the following conjecture (attributed also to Chowla) holds
\begin{Conj}\label{ChowlaIII}
Let $a_1,a_2,\cdots,a_m$ be a $k$ distinct integers. Then, as $N \longrightarrow +\infty$,
$$\sum_{n \leq N}\lamob(n+a_1)\lamob(n+a_2)\cdots\lamob(n+a_m)=o(N).$$
\end{Conj}

\noindent{}Conjecture \ref{ChowlaIII} say that the Liouville function is {\textbf{normal}}, that is, 
$$\frac{\Big|\Big\{ j\in [1,N]~~:~~\lamob(j+n)=\epsilon_n,~~~n=1,\cdots,k\Big\}\Big|}{N} \tend{N}{+\infty}\frac1{2^k}.$$
However, Conjecture \ref{ChowlaII}
say that the Liouville function is \textbf{weak normal}, that is, the number of solutions is infinite.\\

Let us bring to the attention of the reader that to the best knowledge of the author there is no connection known between the 
popular Chowla conjecture \ref{ChowlaIII} and the Riemann Hypothesis unless the trivial case $k=1$. We further notice 
that N.Ng in \cite{Ng} proved that under a more strong conjecture (called M\"{o}bius $s$-tuple conjecture), 
the distribution of $M(x+h)-M(x)$ is normal, where $M(x)$ is the Mertens function given by $M(x)=\sum_{n \leq x}\mob(n)$. 
P. Sarnak wrote about his feeling regarding Chowla conjecture \ref{ChowlaIII} \cite{SarnakL}: 
``I don’t know of any reason to be skeptical about Chowla’s
Conjecture, after all if it is false it would indicate some hidden structure in the integers that has not been observed.'' \\ 

Our principal goal in this section is to compute exactly the $L^\alpha$-norms of  
the trigonometric polynomials with Liouville or M\"{o}bius coefficients. 
Of course 
those polynomials are not square flat by our main result, that is,  
Erd\"{o}s conjectures holds for the trigonometric polynomials with Liouville or M\"{o}bius coefficients. 
But, in connection with the Riemann Hypothesis,  we have the following
\begin{Th}\label{RH}
 Assume that for any $\alpha \geq 1$, we have
$$\Big\|\frac{1}{\sqrt{N}} \sum_{j=1}^{N}\lamob(j) z^j\Big\|_\alpha <+\infty,$$
and 
$$\Big\|\frac{1}{\sqrt{N}} \sum_{j=1}^{N}\mob(j) z^j\Big\|_\alpha <+\infty.$$
Then the Riemann Hypothesis holds.
\end{Th}
\begin{proof} Assume that the Riemann Hypothesis does not holds. Then, according to Littlewood's theorem \cite[p.371]{Titchmarsh} 
(see also \cite[p.261]{Edwards}), there exist $c>0$ and $\epsilon>0$ such that for infinitely many positive integers $N$, we have
$|M(N)| \geq c. N^{\frac{1}{2}+\epsilon}.$ Let $\alpha>1$ such that $\alpha \epsilon>1$. Then, by Marcinkiewicz-Zygmund inequalities,
we have
\begin{eqnarray*}
 \Big\|\frac{1}{\sqrt{N}} \sum_{j=1}^{N}\mob(j) z^j\Big\|_\alpha^\alpha &\geq& A_\alpha 
 \frac{\Big|\ds \sum_{j=1}^{N}\mob(j)\Big|^\alpha}{N^{\frac{\alpha}{2}+1}}\\
&\geq& C_\alpha.\frac{N^{\frac{\alpha}{2}+\alpha\varepsilon}}{N^{\frac{\alpha}{2}+1}}=N^{\alpha\varepsilon-1},
\end{eqnarray*}
Letting $N \longrightarrow +\infty$, we conclude that 
$$\Big\|\frac{1}{\sqrt{N}} \sum_{j=1}^{N}\mob(j) z^j\Big\|_\alpha \tend{N}{+\infty} +\infty.$$
Applying Batman-Chowla trick, the same conclusion can be drawn for $\lamob$. But, we can also give a direct proof. Indeed,
 assume that for any $\alpha \geq 1$, we have
$$\Big\|\frac{1}{\sqrt{N}} \sum_{j=1}^{N}\lamob(j) z^j\Big\|_\alpha <+\infty.$$
Then, by Marcinkiewicz-Zygmund inequalities, for any $\alpha>1$, there exist $A_\alpha$ such that 
$$A_\alpha 
 \frac{\Big|\ds \sum_{j=1}^{N}\lamob(j)\Big|^\alpha}{N^{\frac{\alpha}{2}+1}} \leq c_\alpha,$$
where $c_\alpha$ is some positive constant. This gives 
$$\Big|\ds \sum_{j=1}^{N}\lamob(j)\Big| \leq C_\alpha N^{\frac{1}{2}+\frac{1}{\alpha}}.$$
Since $\alpha$ is arbitrary, it follows, with the help of \eqref{RH2}, that RH holds.   
This accomplishes the proof of the theorem.    
\end{proof}

We will now investigate the flatness issue in the case of polynomials with Liouville and M\"{o}bius coefficients under the assumption that 
Chowla conjecture \ref{ChowlaIII} holds. More precisely,  we have the following result.
\begin{Th}\label{Chowla}
 Assume that Chowla conjecture \ref{ChowlaIII} holds and let $(X_{\lamob},S,\P)$ be the subshift generated by 
the Liouville function. Then, for any $p \geq 1$, we have 
$$\int \Big\|\frac{1}{\sqrt{N}}\sum_{j=1}^{N}x_j z^j\Big \|_p^p d\P(x) \tend{N}{+\infty}\Gamma\Big(\frac{p}{2}+1\Big),$$
in particular
$$\int \Big\|\frac{1}{\sqrt{N}}\sum_{j=1}^{N}x_j z^j \Big\|_4^4 d\P(x) \tend{N}{+\infty}\frac{3 \sqrt{\pi}}{4}.$$
\end{Th}

\begin{proof}We proceed in the same manner as in the dynamical proof (proof of Theorem \ref{pairwise-ind}). Assume that 
Chowla conjecture \ref{ChowlaIII} holds. Then, by Sarnak's theorem \cite[p.10]{SarnakL}, it follows that $(X_{\lamob},S,\P)$ 
is a Benouilli system and $\lamob$ is a generic point (see also  Corollary 4.9 from \cite{elabKLR}). We can 
thus apply the rotated Central Limit theorem of   
Peligrad-Wu's from \cite{Wu} (see also \cite{Cohen-Conze}) to obtain that $R_q(x,z)\setdef \frac{1}{\sqrt{N}}\sum_{j=1}^{N}x_j z^j$ converge in distribution under 
$dz \otimes \P$ to the complex Gaussian  distribution ${\mathcal{N}}_{\mathbb {C}}(0,\sigma)$ on ${\mathbb {C}}$ with variance 
$\sigma^2=1$, that is,
$$ (dz \otimes\P)\Big\{ R_N \in A \Big\} \tend{N}{+\infty}\int_A \frac1{\pi}e^{-|\xi|^2} d\xi.$$
Denote by $\mathcal{D}(R_q(x,z))$ the distribution of $R_q(x,z)$ under $dz \otimes \P$. Since
$||R_N||_{L^2(dz \otimes \P)} \leq 1,$ the random variables $R_N(x,z)$ are $L^p$-uniformly integrable, $p \geq 1$, 
by Marcinkiewicz-Zygmund inequality \cite{MZproba} or by Khintchine inequalities \cite[Chap. V,pp.213]{Zygmund}. This combined 
with the Standard Moment Theorem (SMT) \cite[pp.100]{Chung} gives 
\begin{eqnarray*}
 ||R_N(x,z)||_{L^p(dz \otimes \P)}^p &=&\int \big\|R_N(x,z)\big\|_p^p d\P\\
  &=&\int |w|^p d\mathcal{D}(R_N(x,z)) \\   
&&\tenddw[0.5cm]{ q \longrightarrow +\infty } 
\\ 
&&\int |w|^p \frac1{\pi}e^{-|w|^2} dw 
\end{eqnarray*}
where $dw$ is the usual Lebesgue measure on ${\mathbb {C}}$, that is, $dx\cdot dy=rdrd\theta$. To complete the proof, 
it suffices to remark that we have 
\begin{eqnarray*}
\int |w|^p \frac1{\pi}e^{-|w|^2} dw&=&2\int_{0}^{+\infty} r^{p+1}e^{-r^2}dr\\
&\stackrel{s=r^2}{=}&\int_{0}^{+\infty} s^{\frac{p}{2}} e^{-s}ds\\
&=&\Gamma\Big(\frac{p}{2}+1\Big). 
\end{eqnarray*} 
We remind that the gamma function $\Gamma$ is defined by 
$$\Gamma(z)=\int_{0}^{+\infty}x^{z-1}e^{-x} dx,\quad \quad \textrm{with} \quad \re(z)>0.$$
We thus obtain a new proof of theorem of Borwein-Lockhart \cite{Browein-Lo} and end the proof of the theorem.

\end{proof}

Combining Theorem \ref{RH} and Theorem \ref{Chowla}, we obtain 

\begin{Cor}If Chowla conjecture \ref{ChowlaIII} holds then RH holds.
\end{Cor}
 \begin{proof}
 We proceed par contradiction. So suppose that Chowla Conjecture \ref{ChowlaIII} holds and  
 RH does not holds. Then, by Theorem \ref{RH}, there is $\alpha \geq 1$ and subsequence $(N_n)_{n \geq 1}$ 
 (which we still denoted by $N$ for simplicity) such that we have
 $$\Big\|\frac{1}{\sqrt{N}} \sum_{j=1}^{N}\mob(j) z^j\Big\|_\alpha \tend{N}{+\infty} +\infty.$$
This combined with a standard argument (see for example \cite{Lesigne}), yields
$$\int \Big\|\frac{1}{\sqrt{N}}\sum_{j=1}^{N}x_j z^j\Big \|_\alpha^\alpha d\P \tend{N}{+\infty}+\infty,$$
which is impossible in view of Theorem \ref{Chowla} and the proof of corollary is complete.
 \end{proof}

\section{Remarks and open questions}\label{Remarks}
\begin{enumerate}[(i)]
 \item The estimation of $L^4$-norm of Dirichlet kernel $\sqrt{q}D_q$ can be obtained as a corollary of the following more general 
estimation of $L^p$ norm of Dirichlet kernel established in \cite[Lemma 3]{Anderson}
$$\big\|\sqrt{q}.D_q\big\|_p^p{~~\equi{+\infty}~~} \Big(\frac{2}{\pi}\int_{0}^{+\infty}
\Big|\frac{\sin(x)}{x}\Big|^p dx\Big).q^{p-1},~~~~~~~p>1.$$
Indeed, for $p=4$, we have
$$\int_{0}^{+\infty}\Big|\frac{\sin(x)}{x}\Big|^4 dx=\frac{\pi}{3}.$$
\item  Notice that the argument $(GJJH)$ generalized Jensen-Jensen and H{\o}holdt's result (Proposition \ref{JJH}.).
We further obtain in our last argument a generalization of this result to the high degree.
\item Some of our arguments are valid if we assume that $(P_q)$ is $L^\alpha$-flat, $\alpha>1$. However, we need some new ideas
to tackle the problem of $L^\alpha$-flatness in the class of Littlewood, for any $\alpha>1$. As a consequence, it is natural to ask the following question  about the constant $B_\alpha$ 
in the Marcinkiewicz-Zygmund inequalities

$$\textrm{~~Is~~} \frac{2B_\beta^{\alpha-1}}{\pi}\int_{0}^{+\infty}\Big|\frac{\sin(x)}{x}\Big|^\alpha dx>1?$$
where $\beta$ is the conjugate of $\alpha$.\\

This question seems to be related the problem about the sharp constants in the Marcinkiewicz-Zygmund inequalities raised 
in \cite{Lubinsky}.
\item Let $(\eta_j)_{j \geq 0} \in \{0,1\}^\N$. Then, according to P. Cohen's heuristic argument, 
if the sequence $\Big(\ds \int \Big|\sum_{j=0}^{q-1}\eta_j z^j\Big| dz\Big)$ satisfies  
$\ds \int \Big|\sum_{j=0}^{q-1}\eta_j z^j\Big| dz\sim c \sqrt{m_q}$, where  $m_q=\sum_{j=0}^{q-1}\eta_j$, and $c$ is absolute constant, 
then the density of $(\eta_j)_{j \geq 0}$ is zero, that is, the set $\big\{j~~:~~\eta_j=1\big\}$ has density zero. 
The answer to this question is negative. Indeed, 
applying Fukuyama's construction \cite{Fukuyama}, it can be shown that there is a sequence 
$(\eta_j)_{j \geq 0}$ with density $\frac{1}{2}$ and for which the Salem-Zygmund CLT holds.  However, in our setting, 
if we  assume that the sequence of Littlewood polynomials 
is almost everywhere flat then the associated sequence $(\eta_j)$ verify 
$$\Big|\frac{1}{\sqrt{m_q}}\sum_{j=0}^{q-1}\eta_jz^j\Big| \xrightarrow[q \rightarrow +\infty]{\textrm{a.e.}} \frac{\sqrt{2}}{2}.$$
Furthermore, it is proved in \cite{elabdal-little} that if the sequence $(\eta_j)$ generated a flat analytic polynomials in the a.e.
sense the its density is zero.\\

According to this, let us introduce the following notion 
\begin{Def}\label{def1}
A sequence $f_n, n=1.2.\cdots$ of complex valued functions on $S^1$ is said to be $c$-flat a.e. ($dz$) if the sequence of $\mid f_n\mid, n=1,2,\cdots$ of its absolute values  
converges to $c$ a.e. ($dz$). A 1-flat sequence (a.e. ($dz$)) is called  a.e. ($dz$) flat sequence. In case convergence to $c$ is uniform we say that the sequence is $c$- ultraflat. 
A 1-ultraflat sequence is simply called ultraflat.
\end{Def}
Obviously the Dirichlet kernel allows us to produce a sequence of $0$-flat a.e. polynomials. In \cite{elabdal-Banach},  
it is shown  that there is $1$-flat a.e. polynomials from the class of Newman-Bougrain. Those sequences
of polynomials are the only known $c$-flat polynomials from the Newman-Bougrain's class. A natural question is to ask if there is a sequence of polynomials form the Newman-Bourgain class $c$-flat a.e. with $c \in ]0,1[$.
\item Notice that in our proof the square flatness implies that there is a subset $H$  
of non-negative integer with density $\frac{1}{2}$ and for which for any 
$r \geq 3$, $H \mod r=\big\{0,\cdots,r-1\big\}$ and it is equipped with uniform distribution. One may ask if such subset exists. 
The answer is yes. Indeed, let $T_\alpha : \theta \in S^1 \mapsto \theta+\alpha \mod 1,$
$\alpha$ is irrational. Let $H$ be the sequence of return time of $0$ to $(0,\frac{1}{2}]$ under $T_\alpha$. Obviously the density of 
$H$  is $\frac12$. Thanks to the ergodic theorem of Birkhoff. It is also easy to see that $H$ does the job. Thanks to the 
Furstenberg's disjointness of $T_\alpha$ from any rational rotation. Form this, we see  
that for any $d \in ]0,1[$ there exists a subset $H$ of density $d$ and for which for any 
$r \geq 3$, $H \mod r=\big\{0,\cdots,r-1\big\}$ and it is further equipped with uniform distribution. One may ask if it is possible 
to give a topological dynamical proof of this fact, specially for the case $H$ of density zero. 
\item Let us make a connection between our result on Chowla conjecture \ref{ChowlaIII}, RH and 
the heuristic argument of Denjoy \cite{Donjoy}. Let $\varepsilon$ be a positive integer and let $\alpha>1$ such that $\alpha \varepsilon>1$. Assume 
that Chowla conjecture \ref{ChowlaIII} holds. Then,
again, by Marcinkiewicz-Zygmund inequality \cite{MZproba} or by Khintchine inequalities \cite[Chap.V,pp.213]{Zygmund}, we have that 
$R_N(x)=\sum_{j=1}^{N}\pi_0 \circ S^j(x)$ is uniformly $L^\alpha$-integrable. We further have that $R_N(x)$  converge in 
distribution to the normal distribution by the classical CLT. Thus once again we can apply the Standard Moment Theorem to get that
$$N^{\alpha \varepsilon}\Big\|\frac{R_N(x)}{N^{\frac{1}{2}+\varepsilon}}\Big\|_\alpha^{\alpha} \sim 
\frac{2^{\frac{\alpha}{2}}}{\sqrt{\pi}}\Gamma\Big(\frac{\alpha}{2}\Big).$$
It follows that 
$$\sum_{N \geq 1} \Big\|\frac{R_N(x)}{N^{\frac{1}{2}+\varepsilon}}\Big\|_\alpha^{\alpha}<+\infty.$$ 
Hence 
$$\sum_{N \geq 1} \frac{\big|\sum_{j=1}^{N}x_j\big|}{N^{\frac{1}{2}+\varepsilon}} \in L^\alpha(X_{\lamob},\P).$$
We thus conclude that for almost all $x \in X_{\lamob}$,
$$\Big|\sum_{j=1}^{N}x_j\Big| \ll N^{\frac{1}{2}+\epsilon},$$
That is, almost all point are ``good'' in the sense of RH.
\end{enumerate}

\begin{thank}
The author wishes to express his thanks to Fran\c cois Parreau, Jean-Paul Thouvenot, Mahendra Nadkarni, 
for many stimulating conversations on the subject and their sustained interest and encouragement in this work. The author is specialy
indebted to Mahendra Nadkarni for pointing to him a gap in the previous draft and for bring to his attention some deep remarks. 
He further wishes to express his thanks  to William Veech 
and Herv\'e Queffelec for their encouragement and interest in this work. He would like also to thanks the Erwin Schr\"{o}dinger Institute and the organizers of the workshop
``the Workshop on Normal Numbers: Arithmetic, Computational and Probabilistic Aspects” for the invitation.
\end{thank}


  \end{document}